\documentclass[11pt, oneside]{amsart}

\usepackage{etex}
\usepackage[T1]{fontenc}
\usepackage[utf8]{inputenc}
\usepackage[english]{babel}
\usepackage{geometry}

\usepackage{amsmath}
\usepackage{amsfonts}
\usepackage{amssymb}
\usepackage{amsthm}
\usepackage{booktabs}
\usepackage{paralist}
\usepackage{subfig}
\usepackage{array}
\usepackage{xy}
\usepackage{fancyhdr}
\usepackage{wrapfig}
\usepackage{bm}
\usepackage{listings}
\usepackage{enumerate}
\usepackage{multirow}
\usepackage{placeins}
\usepackage{pgfplots}
\usepackage{pgfplotstable}
\pgfplotsset{compat=1.16,width=.7\textwidth}
\usepackage{tikz}
\usepackage{xcolor}
\usepackage{mathtools}
\usepackage{algorithm}
\usepackage{algpseudocode}
\usepackage{dsfont}
\usepackage{stmaryrd}
\usepackage{cite}		
\usepackage{xparse}
\usepackage{siunitx}  

\usepackage{hyperref}

\usepackage{cleveref}

\theoremstyle{definition}
\newtheorem{theorem}{Theorem}[section]

\newtheorem{proposition}[theorem]{Proposition}
\newtheorem{corollary}[theorem]{Corollary}
\newtheorem{definition}[theorem]{Definition}

\newtheorem{remark}[theorem]{Remark}

\crefname{equation}{}{}
\crefname{theorem}{Theorem}{Theorems}
\crefname{lemma}{Lemma}{Lemmas}
\crefname{corollary}{Corollary}{Corollaries}
\crefname{proposition}{Proposition}{Propositions}
\crefname{remark}{Remark}{Remarks}
\crefname{section}{Section}{Sections}
\crefname{figure}{Figure}{Figures}
\crefname{algorithm}{Algorithm}{Algorithms}
\crefname{table}{Table}{Tables}
\crefname{subsection}{Section}{Sections}

\DeclarePairedDelimiter{\norm}{\lVert}{\rVert}

\DeclareMathOperator{\vspan}{span} 	

\newcommand{\C}{\mathbb{C}}
\newcommand{\R}{\mathbb{R}}

\newcommand{\N}{\mathbb{N}}
\newcommand{\nostro}{RKcompress}

\renewcommand{\vec}[1]{\boldsymbol{#1}}		

\newcommand{\mat}[1]{\mathbf{#1}}		

\newcommand{\kryl}{\mathcal{K}}		
\newcommand{\rat}{\mathcal{Q}}		
\newcommand{\poly}{\Pi}				

\makeatletter
\NewDocumentCommand{\LineComment}{s m}{%
  \Statex \IfBooleanF{#1}{\hspace*{\ALG@thistlm}}\(\triangleright\) #2}
\makeatother

\graphicspath{{./}}

\makeatletter
\newcommand{\algmargin}{\the\ALG@thistlm}
\makeatother
\algnewcommand{\parState}[1]{\State%
	\parbox[t]{\dimexpr\linewidth-\algmargin}{\strut #1\strut}}

\makeatletter
\renewcommand*\env@matrix[1][*\c@MaxMatrixCols c]{%
  \hskip -\arraycolsep
  \let\@ifnextchar\new@ifnextchar
  \array{#1}}
\makeatother

\numberwithin{equation}{section}

\title[Low-memory Lanczos with rational Krylov compression]{A low-memory Lanczos method with rational Krylov compression for matrix functions}

\author{Angelo A. Casulli}
\address{Scuola Normale Superiore, Piazza dei Cavalieri 7, 56126 Pisa, Italy}
\email{angelo.casulli@sns.it}

\author{Igor Simunec}
\address{Scuola Normale Superiore, Piazza dei Cavalieri 7, 56126 Pisa, Italy}
\email{igor.simunec@sns.it}

\begin{document}

\maketitle

\begin{abstract}
	In this work we introduce a memory-efficient method for computing the action of a Hermitian matrix function on a vector. Our method consists of a rational Lanczos algorithm combined with a basis compression procedure based on rational Krylov subspaces that only involve small matrices. The cost of the compression procedure is negligible with respect to the cost of the Lanczos algorithm. This enables us to avoid storing the whole Krylov basis, leading to substantial reductions in memory requirements. 
	This method is particularly effective when the rational Lanczos algorithm needs a significant number of iterations to converge and each iteration involves a low computational effort. This scenario often occurs when polynomial Lanczos, as well as extended and shift-and-invert Lanczos are employed.
	Theoretical results prove that, for a wide variety of functions, the proposed algorithm differs from rational Lanczos by an error term that is usually negligible. 
	The algorithm is compared with other low-memory Krylov methods from the literature on a variety of test problems, showing competitive performance.
\end{abstract}

\section{Introduction}

A fundamental problem in numerical linear algebra is the approximation of the action of a matrix function $f(A)$ on a vector $\vec b$, where $A \in \C^{n \times n}$ is a matrix that is typically large and sparse, $\vec b \in \C^n$ is a vector and $f$ is a function defined on the spectrum of $A$. In this work, we focus on the case of a Hermitian matrix $A$. We recall that when $A$ is Hermitian, given an eigendecomposition $A = U D U^H$, the matrix function $f(A)$ is defined as $f(A) = U f(D) U^H$, where~$f(D)$ is a diagonal matrix obtained by appliying $f$ to each diagonal entry of $D$. We refer to \cite{Higham08functions} for an extensive discussion of matrix functions.

Popular methods for the approximation of $f(A) \vec b$ are polynomial~\cite{saad92analysis,hochbruck97krylov,frommer08matrix,frommer14efficient,guettel21comparison} and rational Krylov methods~\cite{druskin98extended,moret04rd,guettel13rational,aceto19rational,benzi22rational}.
The former only access $A$ via matrix-vector products, while the latter require the solution of shifted linear systems with $A$. 
When the linear systems can be solved efficiently, rational Krylov methods can be more effective than polynomial Krylov methods since they usually require much fewer iterations to converge. 
However, there are several situations in which rational Krylov methods are not applicable, either because the matrix~$A$ is only available implicitly via a function that computes matrix-vector products, or because $A$ is very large and the solution of linear systems is prohibitively expensive. 

When $A$ is Hermitian, the core component of a polynomial Krylov method is the Lanczos algorithm~\cite{lanczos50iteration}, which constructs an orthonormal basis $\mat Q_m = [\vec q_1 \: \dots \: \vec q_m]$ of the polynomial Krylov subspace $\kryl_m(A, \vec b) = \vspan\{ \vec b, A \vec b, \dots, A^{m-1} \vec b \}$ by exploiting a short-term recurrence. The product~$f(A) \vec b$ can then be approximated by the Lanczos approximation
\begin{equation*}
	\vec f_m := \mat Q_m f(\mat T_m) \vec e_1 \norm{\vec b}_2, \qquad \mat T_m :=  \mat Q_m^H A \mat Q_m,
\end{equation*} 
where $\vec e_1$ denotes the first unit vector. 
The Lanczos algorithm uses a short-term recurrence in the orthogonalization step, so each new basis vector is orthogonalized only against the last two basis vectors, and only three vectors need to be kept in memory to compute the basis $\mat Q_m$. Even though the basis constructed by the Lanczos algorithm is orthonormal in exact arithmetic, the one computed in finite precision arithmetic will exhibit a loss of orthogonality, leading to a delay in the convergence of the Lanczos approximation of $f(A) \vec b$. For more information on the finite precision behavior of the Lanczos algorithm, we direct the reader to \cite[Section~10.3]{golub2013matrix} and the references therein.
Although the basis $\mat Q_m$ and the projected matrix $\mat T_m$ can be computed by using the short-term recurrence that only requires storage of the last three basis vectors, forming the approximate solution $\vec f_m$ still requires the full basis $\mat Q_m$. 
When the matrix $A$ is very large, there may be a limit on the maximum number of basis vectors that can be stored, so with a straightforward implementation of the Lanczos method there is a limit on the number of iterations of Lanczos that can be performed and hence on the attainable accuracy. 
In the literature, several strategies have been proposed to deal with low memory issues. See the recent surveys~\cite{guettel20limited,guettel21comparison} for a comparison of several low-memory methods.

A simple but effective approach is the two-pass Lanczos method~\cite{borici00fast,frommer08matrix}. With this approach, the Lanczos method is first run once to determine the projected matrix $\mat T_m$ and compute the short vector $\vec t_m = f(\mat T_m)\vec e_1 \norm{\vec b}_2$. 
After $\vec t_m$ has been computed, the Lanczos method is run for a second time to form the product $\vec f_m = \mat Q_m \vec t_m$ as the columns of $\mat Q_m$ are computed. 
This method doubles the number of matrix-vector products with $A$ with respect to standard Lanczos, but it requires storage of only the last three basis vectors.

Another possibility is the multi-shift conjugate gradient method~\cite{vandeneshof02numerical,frommer99fast}. This method is based on an explicit approximation of $f$ with a rational function $r$ expressed in partial fraction form. 
Then, $f(A) \vec b \approx r(A) \vec b$ is approximated by using the conjugate gradient method to solve each of the linear systems that appear in the partial fraction representation of~$r(A) \vec b$. 
This can be done efficiently by exploiting the shift invariance of Krylov subspaces, i.e., the fact that $\kryl_m(A, \vec b) = \kryl_m(A + \theta I, \vec b)$ for any $\theta \in \R$, in order to use a single Krylov subspace to obtain approximate solutions to all the linear systems symultaneously. Compared to Lanczos, this method requires performing additional vector operations and storing vectors proportionally to the number of poles of the rational approximant~$r$.

When $f$ is a Stieltjes function, it is possible to use a restarting strategy that is similar to restarted Krylov subspace methods for linear systems~\cite{eiermann06restarted,frommer14efficient,frommer14convergence}. By exploiting the Stieltjes integral representation of the function $f$, the error after a certain number of Lanczos iterations can be written as $f(A) \vec b - \vec f_m = f_m(A) \vec q_{m+1}$, where $f_m$ is still a Stieltjes function that depends on $f$ and $\mat T_m$. This property makes it possible to restart the Lanczos method after a certain number of iterations, and then iteratively approximate the error using the same method.

Recently, a low-memory method has been proposed to compute an approximation from a Krylov subspace to $f(A) \vec b$ when $f$ is a rational function~\cite{chen23lowmemory}. This approximation is optimal in a norm that depends on the denominator of the rational function. If $d$ is the degree of the denominator of $f$, the approximation from $\kryl_k(A, \vec b)$ can be computed with $k + d$ matrix-vector products with $A$, while storing approximately $2d$ vectors. This method can be extended to non-rational functions $f$ by means of rational approximations, and it often produces approximations that are comparable or better than the Lanczos approximation. 

In this work we propose a new low-memory algorithm for the approximation of $f(A) \vec b$. 
Our method combines an outer rational Lanczos iteration with an inner rational Krylov subspace, which is used to compress the outer Krylov basis whenever it reaches a certain size. 
Similarly to \cite{casulli2024computing}, the inner rational Krylov subspace does not involve the matrix $A$, but only small matrices. 
This is a key observation, since constructing a basis of the inner subspace does not require the solution of linear systems with $A$, and hence it is cheap compared to the cost of the outer Lanczos iteration.
The approximate solutions computed by our algorithm coincide with the ones constructed by the outer Krylov subspace method when $f$ is a rational function, and for a general function they differ by a quantity that depends on the best rational approximant of $f$ with the poles used in the inner rational Krylov subspace. 
In order to obtain a meaningful advantage when compressing the basis, the algorithm that we propose should be used when the outer Krylov method requires several iterations to converge, so polynomial Lanczos is an ideal candidate. 
Other potential options are rational Lanczos methods with repeated poles, such as the extended~\cite{druskin98extended} or shift-and-invert Krylov methods~\cite{moret04rd}.

If the outer Krylov basis is compressed every $m$ iterations and the inner rational Krylov subspace has $k$ poles, our approach requires the storage of approximately $m + k$ vectors. Additionally, due to the basis compression, our approximation involves the computation of matrix functions of size at most $(m+k) \times (m+k)$, so the cost does not grow with the number of iterations. This represents an important advantage with respect to the Lanczos method, since when the number of iterations is very large the evaluation of $f$ on the projected matrix can become quite expensive.

The remainder of the paper is organized as follows. In \cref{subsec:notation} we introduce some general notation that we use throughout the manuscript. Some relevant background regarding Krylov subspaces and low-rank updates of matrix functions is reported in \cref{sec:background}. In \cref{sec:algorithm} we describe the memory-efficient algorithm for the computation of $f(A) \vec b$, and in \cref{sec:algorithm-analysis-and-comparison} we analyze it and compare it with other low-memory methods from the literature. 
The algorithm is tested on several problems from selected applications in \cref{sec:numerical-experiments}, showing competitive performance with respect to existing low-memory methods. \cref{sec:conclusions} contains some concluding remarks.

\subsection{Notation}
\label{subsec:notation}

We use bold lowercase letters to denote vectors and bold uppercase letters to denote certain matrices associated with Krylov subspaces. We denote by $I$ the identity matrix and by $\vec e_i$ the $i$-th unit basis vector, with size that can be inferred from the context when not specified. The field of values (or numerical range) of a matrix $A$ is denoted by $\mathbb{W}(A) := \{\vec x^H A \vec x \,:\, \norm{\vec x}_2 = 1 \}$; it coincides with the convex hull of the spectrum of $A$ when $A$ is Hermitian. We denote by $\overline{\C} = \C \cup \{\infty\}$ the extended complex plane. The conjugate transpose of a matrix $A \in \C^{n \times n}$ is denoted by $A^H$. 
We occasionally use the notation~$\vec \xi_m = [\xi_0, \dots, \xi_{m-1}]$ to denote a list of complex numbers; in such a case, we write $\vec \xi_m \subset \overline{\C}$ and $\xi_j \in \vec \xi_m$ to indicate an element of the list. We denote by $\poly_m$ the set of polynomials of degree up to $m$.  
Given a compact set $\mathcal{S} \subset \C$ and a continuous function $f : \mathcal{S} \to \C$, we denote by $\norm{f}_{\mathcal{S}}$ the supremum norm of the function $f$ on $\mathcal{S}$.     

\section{Background}
\label{sec:background}

In this section we recall some results on Krylov subspaces and on low-rank updates of matrix functions that we are going to employ in the derivation of our algorithm in \cref{sec:algorithm}.

\subsection{Krylov subspaces}
\label{subsec:krylov-background}

Rational Krylov subspaces are a very useful tool for the computation of the action of a matrix function of a vector. A broad description of the topic can be found in \cite{guettel2010rational,guettel13rational}. 
\begin{definition}
	Let $A\in \C^{n\times n}$, $\vec b \in \C^{n}$ and let $\vec{\xi}_k$ be a list of poles $[\xi_0,\dots,\xi_{k-1}] \subset \overline{\C}$ which are not eigenvalues of $A$. The associated rational Krylov subspace is defined as
	\begin{equation*}
		\rat(A,\vec b,\vec \xi_k):=\Big\{q(A)^{-1} p(A) \vec b ,\quad  \text { with } p \in \poly_{k-1}\Big\},
	\end{equation*} 
	where 
	\begin{equation*}
		q(x):=\prod_{\xi\in \vec \xi_k \,,\, \xi \neq \infty}(x-\xi).
	\end{equation*}
\end{definition}

An orthonormal basis $\mat Q_k\in \C^{n\times k}$ of $\rat(A,\vec b,\vec \xi_k)$ can be computed with the rational Arnoldi algorithm~\cite[Algorithm~1]{guettel2010rational} introduced by Ruhe in \cite{ruhe1984rational}, that iteratively computes the columns of~$\mat Q_k$ and two matrices $\underline{\mat K}_{k-1}, \underline{\mat H}_{k-1}\in\C^{k \times (k-1)}$ in upper Hessenberg form such that 
\begin{equation}
	\label{eqn:rad0}
	A \mat Q_{k}\underline{\mat K}_{k-1} =\mat Q_{k}\underline{\mat H}_{k-1}.
\end{equation}
In particular, if $\mat Q_k$ is computed by the rational Arnoldi algorithm, for each $i \le k$ its first $i$ columns span the subspace $\rat(A,\vec b, [\xi_0,\dots,\xi_{i-1}])$. We are going to denote by $\mat K_{k-1}$ and $\mat H_{k-1}$ the $(k-1) \times (k-1)$ leading principal blocks of $\underline{\mat K}_{k-1}$ and $\underline{\mat H}_{k-1}$, respectively.        

The projection of $A$ onto $\rat(A, \vec b, \vec \xi_k)$ is given by $\mat T_k = \mat Q_k^H A \mat Q_k$. After computing $\mat Q_k$ and $\mat T_k$, the action of the matrix function $f(A)$ on the vector $\vec b$ can be approximated as
\begin{equation}
	\label{eqn:krylov-matfun-approx}
	f(A) \vec b \approx \mat Q_k f(\mat T_k) \vec e_1 \norm{\vec b}_2. 
\end{equation}
When all the poles in $\vec \xi_k$ are equal to $\infty$, the associated space reduces to a polynomial Krylov subspace, which we denote by $\kryl_m(A, \vec b)$, and the approximation~\cref{eqn:krylov-matfun-approx} reduces to the Lanczos approximation.   

Each iteration of the rational Arnoldi algorithm consists of the solution of a linear system\footnote{The linear system reduces to a matrix-vector product if the considered pole is infinity.} involving the matrix $A-\xi I$ for a pole $\xi\in\vec \xi_{k}$, and an orthogonalization step in which the new vector is orthogonalized with respect to all the columns of $\mat Q_k$ already computed. 
Usually, the solution of the linear systems is a very expensive operation, especially for large~$A$; to overcome this issue, a commonly employed strategy is to choose $\vec \xi_k$ by cycling a small number of poles and  keep in memory a decomposition of the matrix~$A-\xi I$ for each distinct $\xi \in \vec \xi_k$, which allows us to obtain fast solutions of multiple linear systems with the same pole. 

The orthogonalization step can significantly slow down the procedure when a Krylov subspace with a large dimension is being constructed. For a Hermitian matrix~$A$, we can employ the rational Lanczos algorithm based on short recurrences (see \cite[Section~5.2]{guettel2010rational}) and construct the Krylov subspace by using only two vectors in the orthogonalization step as described\footnote{For simplicity, \cref{algorithm:short-recurrence} does not allow poles equal to zero. A way to overcome this issue is given in Remark~\ref{rmk:zero-pole}.} in Algorithm~\ref{algorithm:short-recurrence}. If the latter procedure is employed, the matrices $\underline{\mat K}_{k-1}$ and $\underline{\mat H}_{k-1}$ that satisfy \eqref{eqn:rad0} are of the form 
\begin{equation}\label{eqn:defHK-tridiag}
	\underline{\mat H}_{k-1}=\begin{bmatrix}
		\alpha_1 & {\beta}_{1}\\
		{\beta}_{1}& \ddots & \ddots\\
		&\ddots&\ddots & {\beta}_{k-2}\\
		&& \beta_{k-2}&\alpha_{k-1}\\
		&& &\beta_{k-1}
	\end{bmatrix}\quad \mathrm{and}\quad \underline{\mat K}_{k-1}= \begin{bmatrix}1 \\ &\ddots\\ &&1\\ 0 &\dots& 0\end{bmatrix} + \begin{bmatrix}
		\xi_0^{-1}\\ &\ddots\\ && \xi_{k-1}^{-1}
	\end{bmatrix}\underline{\mat H}_{k-1},
\end{equation} 
where $\alpha_i$ and $\beta_i$ are computed by Algorithm~\ref{algorithm:short-recurrence} for each $i$, and with the convention that the inverse of an infinite pole equals to zero. 
Notice that in each iteration of the rational Lanczos algorithm we need to solve a linear system in which the right-hand side has two columns instead of one (see~\cref{algorithm:iter-short-rec-arnoldi}). 
If the linear systems are solved with a direct method, this does not represent a significant increase in computational cost with respect to a linear system with a single right-hand side, since the matrix factorization has to be computed only once.
Moreover, it has been observed that short recurrences can produce a numerical loss of orthogonality in the columns of $\mat Q_k$ in finite precision arithmetic; see \cite[Section~10.3]{golub2013matrix} for more details.

\begin{figure}[tp]
	\vspace*{-\baselineskip}
	\begin{minipage}{\columnwidth}
		\begin{algorithm}[H]
			\begin{algorithmic}
				\Require{$A \in \C^{n \times n}, \vec b \in \C^{n},\boldsymbol {\xi}_{k+1}= [\xi_0, \dots, \xi_{k}] \subset \overline{\C}$}
				\Ensure{ $\mat Q_{k+1}\in \C^{n\times (k+1)},$ $\alpha_1,\dots,\alpha_k$, $\beta_1,\dots \beta_k  \in \C$}
				\State $\xi_{-1}\gets \infty, \quad \beta_{0}=0, \quad  \vec q_{0}=\vec 0$
				\State $\vec w \gets (I-A/\xi_0)^{-1}\vec b$ \Comment{with the convention $A/\infty=0$}
				\State $\vec q_1\gets \vec w/\norm{\vec w}_2$  

				\For{$j = 1, \dots, k$}
				\State $[\vec q_{j+1},\alpha_j,\beta_j]\gets \mathrm{\cref{algorithm:iter-short-rec-arnoldi}}(A, \vec q_{j-1}, \vec q_j, \xi_{j-2}, \xi_{j-1}, \xi_{j}, \beta_{j-1})$
				\EndFor
				\State $\mat Q_{k+1} \gets [\vec q_1,\dots,\vec q_{k+1}]$
			\end{algorithmic}
			\caption{Rational Lanczos} 
			\label{algorithm:short-recurrence}
		\end{algorithm}
	\end{minipage}
	\begin{minipage}{\columnwidth}
		\begin{algorithm}[H]
			\begin{algorithmic}
				\Require{$A \in \C^{n \times n}, \vec q_{1}, \vec q_2 \in \C^{n},\xi_0, \xi_1,\xi_2\in \overline{\C}, \beta_1\in \C$ }
				\Ensure{ $ \vec q_3\in \C^{n},$ $\alpha_2,\beta_2\in \C$} 
				\State  ${\vec w}\gets A(\vec q_{2}+\vec q_{1}\beta_{1}/\xi_{0})-\vec q_{1}\beta_{1}$\Comment{with the convention $\beta_1/\infty=0$}
				\State $ {\vec{s}} \gets (I-A/\xi_{1})\vec{q}_2$ \Comment{with the convention $A/\infty=0$}
				\State $[{\vec w}, \vec s]\gets (I-A/\xi_{2})^{-1}[ {\vec w}, {\vec s}]$
				\State $\alpha_2\gets ({\vec w}^H\vec q_2)/(\vec s^H \vec q_2)$
				\State $\vec w \gets {\vec w} -  \vec s \alpha_2$
				\State $\beta_2 \gets \norm{\vec w}_2$
				\State $\vec q_{3}\gets\vec w /\beta_2$
			\end{algorithmic}
			\caption{Single iteration of rational Lanczos} 
			\label{algorithm:iter-short-rec-arnoldi}
		\end{algorithm}
	\end{minipage}
\end{figure}

Another important ingredient in Krylov methods is the computation of the  projected matrix $\mat T_k = \mat Q_k^H A \mat Q_k$.
Computing $\mat T_k$ using its definition requires performing additional matrix-vector products involving $A$ and therefore it is usually avoided. Information about~$\mat T_k$ can be obtained by exploiting the Arnoldi relation \eqref{eqn:rad0} as already shown in \cite{casulli2023computation, palitta2022short}. If we assume that $\xi_{k-1}=\infty$, the last row of $\underline{\mat K}_{k-1}$ is zero by definition, and the matrix $\mat K_{k-1}$ is invertible (see \cite[Lemma~5.6]{guettel2010rational}). Therefore, writing 
 \begin{equation*}
	\mat T_k=\begin{bmatrix}
		\mat T_{k-1} & \vec b_k\\
		\vec b_k^H &*
	\end{bmatrix},
 \end{equation*}
 with $\mat T_{k-1}\in \C^{(k-1)\times (k-1)}$ and $\vec b_k\in \C^{k-1}$, and multiplying \eqref{eqn:rad0} on the right by $\mat K_{k-1}^{-1}$ and on the left by $\mat Q_{k}^H$, we obtain
 \begin{equation}\label{eqn:proj-matrix}
	\mat T_{k-1} = \mat H_{k-1} \mat K_{k-1}^{-1} \quad \text{and} \quad \vec b_k =  {\beta}_{k-1} \mat K_{k-1}^{-H} \vec e_{k-1}.
 \end{equation}
 If we also assume that there exists $s < k$ such that $\xi_{s}=\infty$, then the matrix $\mat K_k$ has the block triangular structure
 \begin{equation*}
	\mat K_{k-1} = \begin{bmatrix}
		\mat K_s & \xi_{s-1}^{-1}  {\beta}_{s-1} \vec e_{s}\vec e_1^H \\
		& \widehat {\mat K}_{k-1}
	\end{bmatrix}, \quad \text{and} \quad \mat K_{k-1}^{-1} = \begin{bmatrix}
		\mat K_s^{-1} & -\xi_{s-1}^{-1} {\beta}_{s-1} \mat K_s^{-1}\vec e_{s}\vec e_1^H\widehat {\mat K}_{k-1}^{-1} \\
		& \widehat {\mat K}_{k-1}^{-1}
	\end{bmatrix}
 \end{equation*}
 with $\mat K_s\in \C^{s\times s}$ and $\widehat{\mat K}_{k-1} \in \C^{(k-1-s) \times (k-1-s)}$. Since $\mat T_{k-1}$ is Hermitian and $\vec e_1^H \widehat{\mat K}_{k-1} = \vec e_1^H$, exploiting \eqref{eqn:proj-matrix} we have 
 \begin{equation}
	\label{eqn:projected-matrix-identity--infinite-pole}
	\mat T_{k-1}=\begin{bmatrix}
		\mat T_{s} & \vec b_{s+1} \vec e_1^H   \\
		\vec e_1 \vec b_{s+1}^H & \widehat {\mat T}_{k-1}
	\end{bmatrix}, \quad \widehat {\mat T}_{k-1}=   \widehat {\mat H}_{k-1}\widehat {\mat K}_{k-1}^{-1} - \xi_{s-1}^{-1} \beta_{s-1}^2 (\vec e_{s}^H \mat K_s^{-1}\vec e_{s})\vec e_1\vec e_1^H,
 \end{equation} 
 with $\mat T_{s}\in \C^{s\times s}$ and $\vec b_{s+1}\in \C^{s}$ defined according to \eqref{eqn:proj-matrix}. Note that only the $(1, 1)$ entry of $\widehat{\mat T}_{k-1}$ is different from the corresponding element of $\widehat{\mat H}_{k-1} \widehat{\mat K}_{k-1}^{-1}$. 

\begin{remark}\label{rmk:zero-pole}
	For simplicity, Algorithm~\ref{algorithm:short-recurrence} does not allow poles equal to zero. A way to overcome this issue is to choose $\mu \notin \vec \xi_k$ and employ Algorithm~\ref{algorithm:short-recurrence} to compute an orthonormal basis of the Krylov subspace $\rat(A-\mu I,\vec b, [\xi_i-\mu]_{i = 0}^{k-1})$ which is mathematically equivalent to $\rat(A,\vec b, \vec \xi_{k})$, but does not contain any poles equal to zero. We can then recover the projection of $A$ by adding~$\mu I$ to the projected matrix computed by exploiting the Arnoldi relation~\eqref{eqn:rad0} as shown above.
\end{remark}

We conclude this section with a simple result that will be useful in \cref{subsec:low-rank-updates}.

\begin{proposition}\label{prop:krylov_larger}
	Let $A\in \C^{n \times n}$, $\vec b\in \C^{n}$ and let $\vec \xi_k $ be a list of $k$ poles. Denoting by $\mat Q_k$ an orthonormal basis of $\rat(A,\vec b,\vec \xi_k)$, for any matrix $U$ with orthonormal columns such that $\text{span}(\mat Q_k)\subseteq \text{span}(U)$ we have 
	\begin{equation*}
		U\rat(U^HAU,U^H \vec b,\vec \xi_k)= \rat(A,\vec b,\vec \xi_k).
	\end{equation*}
\end{proposition}
\begin{proof}
	We first prove this result in the polynomial case, that is when all the poles are equal to infinity, and then we generalize the proof to the rational case.

	We are going to prove by induction on $h\le k$ that  
	\begin{equation*}
		U(U^HAU)^h U^H\vec b =UU^HA^h \vec b,
	\end{equation*} 
	that in particular implies
	\begin{equation}\label{eqn:in-lemma}
		U\sum_{j=0}^k \gamma_j (U^HAU)^jU^H\vec b = UU^H\sum_{j=0}^k \gamma_j A^j\vec b,
	\end{equation}
	for any choice of coefficients $\gamma_j \in \C$. Note that this is a slightly stronger result that, since $\text{span}(\mat Q_k)\subseteq \text{span}(U)$, for $h < k$ reduces to the thesis in the polynomial Krylov case. 	
	If $h=0$ the statement is straightforward. For a fixed $1 \le h \le k$, we have 
	\begin{equation*}
		U(U^HAU)^hU^H\vec b=UU^HAU(U^HAU)^{h-1}U^H\vec b =UU^HAUU^HA^{h-1}\vec b=UU^HA^h\vec b,
	\end{equation*} 
	where the second equality follows from the inductive hypothesis and the third one is due to  
	\begin{equation*}
		A^{h-1}\vec b \in \kryl_k(A,\vec b) \subseteq \vspan (U).
	\end{equation*}
	For the rational Krylov case, note that 
	\begin{equation*}
		\rat(A,\vec b,\vec \xi_k)= \kryl_k(A, q(A)^{-1}\vec b),
	\end{equation*}
	where $q(x):=\prod_{\xi\in \vec \xi_k, \xi\neq \infty}(x-\xi)$. In particular, denoting by $\vec c:= q(A)^{-1}\vec b$, we deduce from the polynomial case that
	\begin{equation*}
		U\kryl_k(U^HAU,U^H \vec c)= \kryl_k(A, \vec c) = \rat(A,\vec b,\vec \xi_k),
	\end{equation*}
	hence to conclude it is sufficient to prove that $U^H \vec c = q(U^HAU)^{-1}U^H\vec b$. Since the degree of $q$ is bounded by $k$, by \cref{eqn:in-lemma} we have
	\begin{equation*}
		UU^H \vec b = UU^H q(A) \vec c = Uq(U^HAU)U^H\vec c,
	\end{equation*}
	and multiplying both sides on the left by $q(U^HAU)^{-1}U^H$ we conclude.
\end{proof}

\subsection{Low-rank updates of matrix functions}
\label{subsec:low-rank-updates}

In this section we recall a result from \cite{BCKS21} on the computation of a low-rank update of a matrix function, and we slightly modify it to fit our purposes. 

\begin{theorem}
	\label{thm:BCKS}
	\cite[Theorem~3.3]{BCKS21}
	Let $A\in \C^{n\times n}$ be a Hermitian matrix, $B = [\vec b_1 \: \vec b_2] \in \C^{n \times 2}$ and $\Delta \in \C^{2 \times 2}$ Hermitian. Let $r=p/q$ be a rational function with complex conjugate roots and poles, and denote by $\vec {\xi}_k$ the list of cardinality $k:=\max\{\deg(p)+1,\deg(q)\}$ that contains the poles of $r$ and infinity in the remaining elements. We have
	\begin{equation*}
		r(A + B \Delta B^H) - r(A) = \mat Q_k \left(r(\mat Q_k^H(A + B \Delta B^H) \mat Q_k)-r(\mat Q_k^H A \mat Q_k)\right) \mat Q_k^H,
	\end{equation*}
	where $\mat Q_k$ is an orthonormal basis of the subspace $\rat(A,\vec b_1 ,\vec \xi_k) + \rat(A, \vec b_2, \vec \xi_k)$.
\end{theorem}

The following proposition extends \cref{thm:BCKS} to the case where $\mat Q_k$ is replaced by a matrix $U$ such that $\vspan(U) \supseteq \rat(A, \vec b_1, \vec \xi_k) + \rat(A, \vec b_2, \vec \xi_k)$.

\begin{proposition}
	\label{prop:lowrank-update-matfun--larger-basis}
	Let $A, B, \Delta, r$ and  $\vec \xi_k$ be defined as in Theorem~\ref{thm:BCKS}. Then, we have
	\begin{equation*}
		r(A + B \Delta B^H) - r(A) = {U}\left( r({U}^H(A + B \Delta B^H){U})-r({U}^H A {U})\right){U}^H,
	\end{equation*}
	where ${U}$ is a matrix with orthonormal columns such that $\text{span} ({U})\supseteq\rat(A, \vec b_1,\vec \xi_k) + \rat(A, \vec b_2, \vec \xi_k)$.
\end{proposition}

\begin{proof}
	First of all, we note that, since $\text{span} ({U})\supseteq\rat(A,\vec b_1,\vec \xi_k) + \rat(A, \vec b_2, \vec \xi_k)$ we have
	\begin{equation*}
		{U}{U}^H \mat Q_k = \mat Q_k,
	\end{equation*}
	where $ \mat Q_k$ is an orthonormal basis of $\rat(A,\vec b_1,\vec \xi_k) + \rat(A, \vec b_2, \vec \xi_k)$. Moreover, by \cref{thm:BCKS} we have 
	\begin{equation*}
		r(A + B \Delta B^H) - r(A) = \mat Q_k Y_k(r) \mat Q_k^H={U}{U}^H \mat Q_kY_k(r) \mat Q_k^H{U}{U}^H,
	\end{equation*}
	with
	\begin{equation*}
		Y_k(r):=r( \mat Q_k^H(A + B \Delta B^H) \mat Q_k)-r( \mat Q_k^H A \mat Q_k).
	\end{equation*}	
	To conclude it is sufficient to prove that 
	\begin{equation*}
		{U}^H \mat Q_kY_k(r) \mat Q_k^H{U}=r({U}^H(A + B \Delta B^H){U}) - r({U}^H A{U}).
	\end{equation*}
	It follows from Proposition~\ref{prop:krylov_larger} that $U^H \mat Q_k$ is an orthonormal basis of the subspace
	\begin{equation*}
		\rat({U}^HA{U}, U^H \vec b_1, \vec \xi_k) + \rat( U^H A U, U^H \vec b_2, \vec \xi_k) = U^H (\rat(A, \vec b_1, \xi_k) + \rat(A, \vec b_2, \xi_k)),
	\end{equation*}
	and by Theorem~\ref{thm:BCKS} we obtain 
	\begin{align*}
		&r({U}^H(A + B \Delta B^H){U}) - r({U}^HA{U})\\
		&\quad = {U}^H \mat Q_k \left( r( \mat Q_k^H{U}{U}^H(A + B \Delta B^H){U}{U}^H \mat Q_k)-r( \mat Q_k^H{U}{U}^HA{U}{U}^H \mat Q_k) \right) \mat Q_k^H{U}\\
		&\quad= {U}^H \mat Q_k(r( \mat Q_k^H(A + B \Delta B^H) \mat Q_k)-r( \mat Q_k^HA \mat Q_k)) \mat Q_k^H{U}\\
		&\quad= {U}^H \mat Q_kY_k(r) \mat Q_k^H{U}.
	\end{align*}
\end{proof}

The following immediate corollaries are going to be fundamental in the following sections.

\begin{corollary}
	\label{cor:lowrank-update-matfun--eye}
	Assume that $r = p/q$ and $\vec \xi_k$ are defined as in \cref{thm:BCKS} and let $A \in \C^{n \times n}$ be a Hermitian matrix, partitioned as
	\begin{equation*}
		A=\begin{bmatrix}
			A_{11}&\\&A_{22}
		\end{bmatrix} + \begin{bmatrix}
			&A_{12}\\A_{12}^H &
		\end{bmatrix},
	\end{equation*}
	with $A_{11} \in \C^{m \times m}$. Assume that $A_{12}= \vec b \vec c^H$, where $\vec b\in \C^{m}$ and $\vec c\in \C^{n-m}$. Then
	\begin{equation}
		\label{eqn:lowrank-update-matfun}
		r(A)= \begin{bmatrix}
			r(A_{11}) & \\
			& r(A_{22})
		\end{bmatrix} + \begin{bmatrix}
			U_k&\\&I
		\end{bmatrix}X_k(r)\begin{bmatrix}
			U_k^H&\\&I
		\end{bmatrix},
	\end{equation}
	where $U_k$ is an orthonormal basis of $\rat(A_{11},\vec b,\vec \xi_k)$, $I$ is the $(n - m) \times (n - m)$ identity matrix and  
	\begin{equation}
		\label{eqn:lowrank-update-corematrix}
		X_k(r) = r\left(\begin{bmatrix}
			U_k^HA_{11}U_k & U_k^H\vec b\vec c^H \\
			\vec c \vec b^HU_k & A_{22}
		\end{bmatrix}\right) - 
		\begin{bmatrix}
			r(U_k^HA_{11}U_k) & \\
			& r(A_{22})
		\end{bmatrix}.
	\end{equation}
	In particular, for any vector of the form $\begin{bmatrix} \vec v \\ \vec 0\end{bmatrix}$ where $\vec v\in \C^m$, we have
	\begin{equation}
		\label{eqn:lowrank-update-matfun-times-vector}
		r(A)\begin{bmatrix} \vec v \\ \vec 0 \end{bmatrix} = \begin{bmatrix}
			\vec y \\
			\vec 0
		\end{bmatrix} + \begin{bmatrix} \vec w \\ \vec 0\end{bmatrix} + \vec r,
	\end{equation}
	where
	\begin{equation*}
		\vec y =  r(A_{11}) \vec v, \qquad \vec w = - U_k r(U_k^H A_{11} U_k) U_k^H \vec v,
	\end{equation*}
	and
	\begin{equation*}
		\vec r = \begin{bmatrix}
			U_k & \\
			& I
		\end{bmatrix} r\left(\begin{bmatrix}
			U_k^HA_{11}U_k & U_k^H\vec b \vec c^H \\
			\vec c \vec b^HU_k & A_{22}
		\end{bmatrix}\right)  \begin{bmatrix}
			U_k^H \vec v \\ 0
		\end{bmatrix}.
	\end{equation*}
\end{corollary}
\begin{proof}
	We can write
	\begin{equation*}
		\begin{bmatrix}
			& A_{12} \\
			A_{12}^H &
		\end{bmatrix} = \begin{bmatrix}
			\vec b & \\
			& \vec c
		\end{bmatrix} \begin{bmatrix}
			0 & 1 \\
			1 & 0
		\end{bmatrix} \begin{bmatrix}
			\vec b^H & \\
			& \vec c^H
		\end{bmatrix}.
	\end{equation*}
	Due to the block diagonal structure, we have
	\begin{equation*}
		\rat \left(\begin{bmatrix}
			A_{11} & \\
			& A_{22}
		\end{bmatrix}, \begin{bmatrix}
			\vec b \\
			0
		\end{bmatrix}, \vec \xi_k \right)  + \rat \left(\begin{bmatrix}
			A_{11} & \\
			& A_{22}
		\end{bmatrix}, \begin{bmatrix}
			0 \\
			\vec c
		\end{bmatrix}, \vec \xi_k \right) = \vspan \left( \begin{bmatrix}
			U_k & \\
			& V_k
		\end{bmatrix} \right),
	\end{equation*}
	where $V_k$ is an orthonormal basis of $\rat(A_{22}, \vec c, \vec \xi_k)$. 
	Since
	\begin{equation*}
		\left( \begin{bmatrix}
			U_k & \\
			& V_k
		\end{bmatrix} \right) \subseteq \vspan \left( \begin{bmatrix}
			U_k & \\
			& I
		\end{bmatrix} \right) \in \C^{n \times (k + n-m)},
	\end{equation*}
	we can use \cref{prop:lowrank-update-matfun--larger-basis} to obtain \cref{eqn:lowrank-update-matfun}, and then \cref{eqn:lowrank-update-matfun-times-vector} follows as an immediate consequence.
\end{proof}

\begin{corollary}
	\label{cor:lowrank-update-matfun--eye--general-f}
	Let $A$ be as in \cref{cor:lowrank-update-matfun--eye} and let $f$ be a function that is analytic on the numerical range $\mathbb{W}(A)$. Given a list of poles $\vec \xi_k \subset \overline{\C}$, let $U_k$ be an orthonormal basis of~$\rat(A_{11}, \vec b, \vec \xi_k)$. We have
	\begin{equation*}
		f(A) - \begin{bmatrix}
			f(A_{11}) & \\
			& f(A_{22})
		\end{bmatrix} - \begin{bmatrix}
			U_k&\\&I
		\end{bmatrix} X_k(f) \begin{bmatrix}
			U_k^H&\\&I
		\end{bmatrix} = E_k(f),
	\end{equation*}
	with 
	\begin{equation}
		\label{eqn:lowrank-update-matfun--general-f-error}
		\norm{E_k(f)}_F \le 4 \min_{r \in \poly_{k-1}/q} \norm{f - r}_{\mathbb{W}(A)},
	\end{equation}
	where $q(x) =\prod_{\xi\in \vec \xi_k \,,\, \xi \neq \infty}(x-\xi).$ 
\end{corollary}
\begin{proof}
	The bound \cref{eqn:lowrank-update-matfun--general-f-error} easily follows from the exactness property of \cref{cor:lowrank-update-matfun--eye} for rational functions, with the same proof as \cite[Theorem~4.5]{BCKS21}. 
\end{proof}

\section{Algorithm description} 
\label{sec:algorithm}

In this section we present a low-memory implementation of a Krylov subspace method for the computation of $f(A) \vec b$ for a Hermitian matrix $A$ and a vector $\vec b$. 
On a high level, we use an outer Krylov subspace, which can be either polynomial or rational, combined with an inner rational Krylov subspace that is employed to compress the outer Krylov subspace basis and reduce the memory usage. The inner Krylov subspace is constructed using the projection of $A$ on the outer Krylov subspace, so it does not involve any expensive operations with the matrix $A$.  
This approach is designed for scenarios where the outer Krylov method requires a large number of iterations to converge, in order to take full advantage of the basis compressions.
The outer iterations should be cheaper than the solution of shifted linear systems involving $A$ and inner poles, otherwise it would be more efficient to simply use a rational Krylov method associated with the inner poles directly on the matrix~$A$. 
Therefore, an ideal choice for the outer subspace is polynomial Krylov, and other viable options are rational Krylov methods with a few repeated poles, such as extended \cite{druskin98extended} and shift-and-invert \cite{moret04rd} Krylov methods. 

The algorithm is composed of $s$ cycles, with each cycle consisting in~$m$ iterations of the outer Krylov method, and a subsequent compression of the basis to $k$ vectors.
At any given moment, the algorithm keeps in memory at most $m + k$ basis vectors and some additional quantities (whose storage is not dependent on $n$), such as projected matrices of size $m+k$.
In total, the algorithm performs $M = k + ms$ iterations of the outer Krylov subspace method, and the approximation computed at the end of each cycle coincides with the approximation computed by a standard implementation of the outer Krylov method, up to an error due to the rational approximation done in the inner Krylov subspace, which is usually negligible. 
This error is zero in the case where $f$ is a rational function of type $(k-1, k)$, so for simplicity we start by describing the algorithm in this special case.  

Let $r$ be a rational function of type $(k-1, k)$, i.e., $r(z) = p(z)/q(z)$ with $p$ and $q$ polynomials of degrees at most $k-1$ and $k$, respectively. We assume that $r$ has complex conjugate roots and poles. We let $\vec \xi_k$ be the list of poles $[\xi_0, \dots, \xi_{k-1}]$ containing the roots of $q$ and $\infty$ in the remaining entries in the case $\deg(q) < k$.  These poles will be used for the inner rational Krylov iteration. 
Our goal is to construct an approximation of $r(A) \vec b$ from the outer  Krylov subspace $\rat(A, \vec b, \vec \theta_M)$, where $M = k + ms$ and $\vec \theta_M$ is the associated list of complex conjugate poles $[\theta_0, \theta_1, \dots, \theta_{M-1}] \subset \overline{\C}$. We assume that $\theta_0 = \infty$, and $\theta_{jm + k} = \infty$, for $j = 1, \dots, s-1$. We denote by $\vec \theta_\ell$ the list $\vec \theta_\ell = [\theta_0, \dots, \theta_{\ell - 1}]$, for $\ell < M$. Throughout this section, we are going to denote by $\mat Q_j$, $\underline{\mat H}_j$, $\underline{\mat K}_j$ and $\mat T_j$ the matrices associated to the outer Krylov subspace $\rat(A, \vec b, \vec \theta_j)$, according to the notation introduced in \cref{subsec:krylov-background}.  

We also introduce some additional notation for the matrices associated with the outer Krylov subspace $\rat(A, \vec b, \vec \theta_M)$, in order to simplify the exposition in this section.
We denote by~$\widehat{Q} := \mat Q_M$ the orthonormal basis of $\rat(A, \vec b, \vec \theta_M)$ and by~$\widehat{T} := \widehat{Q}^HA\widehat Q = \mat T_M$.
We denote the approximation~\cref{eqn:krylov-matfun-approx} to $r(A) \vec b$ from $\rat(A, \vec b, \vec \theta_M)$ by
\begin{equation}
	\label{eqn:final-approximation-yhat}
	\widehat{\vec y} := \widehat{Q} \, r(\widehat{T}) \vec e_1 \norm{\vec b}_2,
\end{equation}   
where $\vec e_1 \in \R^M$ is the first unit basis vector. We split $\widehat{Q} = \begin{bmatrix}
	Q_1 & Q_2 & \dots & Q_s
\end{bmatrix}$, with~$Q_1 \in \C^{n \times (k+m)}$ and $Q_i \in \C^{n \times m}$ for $i = 2, \dots, s$. We also use the notation $\widehat{Q}_i = \begin{bmatrix}
	Q_i & \dots & Q_s
\end{bmatrix}$, so that we have $\widehat{Q}_i = \begin{bmatrix}
	Q_i & \widehat{Q}_{i+1}
\end{bmatrix}$ for $i = 1, \dots, s-1$.

We introduce a similar notation for the matrix $\widehat{T}$. 
We denote by $T_i$ the $i$-th block on the block diagonal of $\widehat{T}$, with $T_1 \in \C^{(k+m) \times (k+m)}$ and $T_i \in \C^{m \times m}$ for $i = 2, \dots, s$, and by $\widehat{T}_{i+1} \in \C^{(s-i)m \times (s-i)m}$ the trailing block on the diagonal after $T_{i}$. So we have
\begin{equation*}
	\widehat{T} = \begin{bmatrix}
		T_1 & \vec b_1 \vec e_1^H \\
		\vec e_1 \vec b_1^H & \widehat{T}_2
	\end{bmatrix}, \qquad \widehat{T}_i = \begin{bmatrix}
		T_i & \vec b_i \vec e_1^H \\
		\vec e_1 \vec b_i^H & \widehat{T}_{i+1}
	\end{bmatrix}, \quad 2 \le i \le s-1,
\end{equation*} 
where $\vec b_1 \in \C^{m+k}$, $\vec b_i \in \C^m$ for $i = 1, \dots, s-1$, $\vec e_1$ is the first unit base vector of the appropriate dimension, and $\widehat{T}_s = T_s$. The block structure of $\widehat{T}$ is a consequence of the poles at infinity in $\vec \theta_M$, as shown in \cref{subsec:krylov-background}.  

\subsection{First cycle}

We have 
\begin{equation*}
	\widehat{T} = \begin{bmatrix}
		T_1 & \\
			& \widehat{T}_2
	\end{bmatrix} + \begin{bmatrix}
		& \vec b_1 \vec e_1^H \\
		\vec e_1 \vec b_1^H &
	\end{bmatrix}.
\end{equation*}
Let $U_1 \in \C^{(m + k) \times k}$ be an orthonormal basis of the rational Krylov subspace~$\rat(T_1, \vec b_1, \vec \xi_k)$.
Using \cref{cor:lowrank-update-matfun--eye}, we can write
\begin{equation*}
	r(\widehat{T}) \vec e_1 \norm{\vec b}_2 = r\left(\begin{bmatrix}
		T_1 & \vec b_1 \vec e_1^H \\
		\vec e_1 \vec b_1^H & \widehat{T}_2
	\end{bmatrix}\right)\begin{bmatrix} \vec e_1 \norm{\vec b}_2 \\ \vec 0\end{bmatrix}= \begin{bmatrix}
		\widetilde{\vec y}_1 \\ \vec 0
	\end{bmatrix} + \begin{bmatrix} \widetilde{\vec w}_1 \\ \vec 0\end{bmatrix} + \widetilde{\vec r}_1,
\end{equation*}
where
\begin{equation*}
	\widetilde{\vec y}_1 = r(T_1) \vec e_1 \norm{\vec b}_2, \qquad \widetilde{\vec w}_1 = - U_1 r(U_1^H T_1 U_1) U_1^H \vec e_1 \norm{\vec b}_2,
\end{equation*}
and
\begin{equation*}
	\widetilde{\vec r}_1 = \begin{bmatrix}
		U_1 & \\
		& I
	\end{bmatrix} 
	r\bigg(\begin{bmatrix}
		U_1^H T_1 U_1 &  U_1^H \vec b_1 \vec e_1^H \\
		\vec e_1 \vec b_1^H U_1 & \widehat{T}_2
	\end{bmatrix} \bigg) 
	\begin{bmatrix}
		U_1^H \vec e_1 \norm{\vec b}_2 \\ 0
	\end{bmatrix}.
\end{equation*}
Recalling the notation introduced above, we have
\begin{equation*}
	\widehat{Q} r(\widehat{T}) \vec e_1 \norm{\vec b}_2 = \begin{bmatrix}
		Q_1 & \widehat{Q}_2 
	\end{bmatrix} \left(\begin{bmatrix}
		\widetilde{\vec y}_1 \\
		\vec 0
	\end{bmatrix}  + \begin{bmatrix}
		\widetilde{\vec w}_1 \\
		\vec 0
	\end{bmatrix} + \widetilde{\vec r}_1 \right) = Q_1 \widetilde{\vec y}_1 + Q_1 \widetilde{\vec w}_1 + \begin{bmatrix}
		Q_1 & \widehat{Q}_2
	\end{bmatrix} \widetilde{\vec r}_1. \\
\end{equation*}
To summarize, we have obtained
\begin{equation}
	\label{eqn:krylov-update--first-iteration-solution-identity}
	\widehat{\vec y} = \widehat{Q} \, r(\widehat{T}) \vec e_1 \norm{\vec b}_2 = \vec y_1 + \vec w_1 + \vec r_1,
\end{equation}
where
\begin{equation*}
	\vec y_1 = Q_1 r(T_1) \vec e_1 \norm{\vec b}_2
\end{equation*}
is the approximation~\cref{eqn:krylov-matfun-approx} to $r(A) \vec b$ from $\rat(A, \vec b, \vec \theta_{k+m})$,
\begin{equation*}
	\vec w_1 = - Q_1 U_1 r(U_1^H T_1 U_1) U_1^H \vec e_1 \norm{\vec b}_2
\end{equation*}
and
\begin{equation}
	\label{eqn:krylov-update--remainder-first-iteration-temporary}
	\vec r_1 = \begin{bmatrix}
		Q_1 U_1 & \widehat{Q}_2
	\end{bmatrix} 
	r\bigg(\begin{bmatrix}
		U_1^H T_1 U_1	& U_1^H \vec b_1 \vec e_1^H \\
		\vec e_1 \vec b_1^H U_1 & \widehat{T}_2
	\end{bmatrix}\bigg) 
	\begin{bmatrix}
		U_1^H \vec e_1 \norm{\vec b}_2 \\ 0
	\end{bmatrix}.
\end{equation}
After the first $k+m$ iterations of the outer Krylov subspace method, we can compute~$\vec y_1$ and~$\vec w_1$, but the remainder $\vec r_1$ still involves $\widehat{Q}_2$ and $\widehat{T}_2$, which have not been computed yet.  
A crucial observation that allows us to save memory is that in \cref{eqn:krylov-update--remainder-first-iteration-temporary} the basis $Q_1$ only appears in the product $Q_1 U_1$, so from now on it is no longer necessary to keep in memory the whole matrix $Q_1$.

Introducing the notation $V_1 := Q_1$, $S_1 := T_1$, $\vec v_1 := \vec e_1 \norm{\vec b}_2$ and $\widetilde{\vec b}_1 := \vec b_1$, we can rewrite \cref{eqn:krylov-update--remainder-first-iteration-temporary} as
\begin{equation}
	\label{eqn:krylov-update--remainder-first-iteration}
	\vec r_1 = \begin{bmatrix}
		V_1 U_1 & \widehat{Q}_2
	\end{bmatrix} 
	r\left(\begin{bmatrix}
		U_1^H S_1 U_1	& U_1^H \widetilde{\vec b}_1 \vec e_1^H \\
		\vec e_1 \widetilde{\vec b}_1^H U_1 & \widehat{T}_2
	\end{bmatrix} \right) 
	\begin{bmatrix}
		U_1^H \vec v_1 \\ 0
	\end{bmatrix}.
\end{equation}
The notation introduced for \cref{eqn:krylov-update--remainder-first-iteration} sets us up for describing the $i$-th cycle of the algorithm.

\subsection{$i$-th cycle}
At the beginning of the $i$-th cycle, with $2 \le i \le s-1$, our task is to compute the remainder $\vec r_{i-1}$ from the previous cycle, which is given by
\begin{equation}
	\label{eqn:krylov-update--remainder-generic-iteration-start}
	\vec r_{i-1} = \begin{bmatrix}
		V_{i-1} U_{i-1} & \widehat{Q}_i
	\end{bmatrix} r \left( \begin{bmatrix}
		U_{i-1}^H S_{i-1} U_{i-1}	& U_{i-1}^H \widetilde{\vec b}_{i-1} \vec e_1^H \\
		\vec e_1 \widetilde{\vec b}_{i-1}^H U_{i-1} & \widehat{T}_i
	\end{bmatrix} \right) \begin{bmatrix}
		U_{i-1}^H \vec v_{i-1} \\ 0
	\end{bmatrix}.
\end{equation}
Note that \cref{eqn:krylov-update--remainder-first-iteration} coincides with \cref{eqn:krylov-update--remainder-generic-iteration-start} with $i = 2$.

Expanding $\widehat{T}_i$ in terms of $T_i$, $\vec b_i$ and $\widehat{T}_{i+1}$, we have   
\begin{align*}
	\begin{bmatrix}
		U_{i-1}^H S_{i-1} U_{i-1} & U_{i-1}^H \widetilde{\vec b}_{i-1} \vec e_1^H \\
		\vec e_1 \widetilde{\vec b}_{i-1}^H U_{i-1} & \widehat{T}_i
	\end{bmatrix} &= \begin{bmatrix}
		U_{i-1}^H S_{i-1} U_{i-1} & U_{i-1}^H \widetilde{\vec b}_{i-1} \vec e_1^H  &  \\
		\vec e_1 \widetilde{\vec b}_{i-1}^H U_{i-1} & T_i & \vec b_i \vec e_1^H \\
		 & \vec e_1 \vec b_i^H & \widehat{T}_{i+1}
	\end{bmatrix}.
\end{align*}
Introducing the notation $V_i := \begin{bmatrix}
	V_{i-1} U_{i-1} & Q_i
\end{bmatrix} \in \C^{n \times (k+m)}$, $\vec v_i := \begin{bmatrix}
	U_{i-1}^H \vec v_{i-1} \\ 0
\end{bmatrix} \in \C^{k+m}$, $\widetilde{\vec b}_i := \begin{bmatrix}
	0 \\ 
	\vec b_i
\end{bmatrix} \in \C^{k+m}$ and
\begin{equation*}
	S_i = 
	\begin{bmatrix}
		U_{i-1}^H S_{i-1} U_{i-1} & U_{i-1}^H \widetilde{\vec b}_{i-1} \vec e_1^H \\
		\vec e_1 \widetilde{\vec b}_{i-1}^H U_{i-1} & T_i
	\end{bmatrix},
\end{equation*}
we can rewrite~\cref{eqn:krylov-update--remainder-generic-iteration-start} as
\begin{equation}
	\label{eqn:krylov-update--remainder-generic-iteration-rewritten}
	\vec r_{i-1} = \begin{bmatrix}
		V_i & \widehat{Q}_{i+1} 
	\end{bmatrix} r\left(\begin{bmatrix}
		S_i	& \widetilde{\vec b}_i \vec e_{1}^H \\
		\vec e_{1} \widetilde{\vec b}_i^H & \widehat{T}_{i+1}
	\end{bmatrix}\right) \begin{bmatrix}
		\vec v_i \\ 0
	\end{bmatrix}.
\end{equation}
The right-hand side of \cref{eqn:krylov-update--remainder-generic-iteration-rewritten} can be computed with the same strategy that was used in the first cycle to compute $\widehat{Q} r(\widehat{T}) \vec e_{1} \norm{\vec b}_2$.
Letting~$U_i \in \C^{(k+m) \times k}$ be an orthonormal basis of $\rat(S_i, \widetilde{\vec b}_i, \vec \xi_k)$, we can use \cref{cor:lowrank-update-matfun--eye} once again to obtain
\begin{equation*}
	r\left(\begin{bmatrix}
		S_i	& \widetilde{\vec b}_i \vec e_1^H \\
		\vec e_1 \widetilde{\vec b}_i^H & \widehat{T}_{i+1}
	\end{bmatrix}\right) \begin{bmatrix}
		\vec v_i \\ 0
	\end{bmatrix} = 
	\begin{bmatrix}
		\widetilde{\vec y}_i \\
		\vec 0
	\end{bmatrix} + 
	\begin{bmatrix}
		\widetilde{\vec w}_i \\
		\vec 0
	\end{bmatrix} + \widetilde{\vec r}_i,
\end{equation*}
where 
\begin{equation*}
	\widetilde{\vec y}_i = r(S_i) \vec v_i, \qquad \widetilde{\vec w}_i = - U_i r(U_i^H S_i U_i) U_i^H \vec v_i,
\end{equation*}
and
\begin{equation*}
	\widetilde{\vec r}_i = \begin{bmatrix}
		U_i & \\
		& I
	\end{bmatrix} r\left(\begin{bmatrix}
		U_i^H S_i U_i & U_i^H \widetilde{\vec b}_i \vec e_1^H \\
		\vec e_1 \widetilde{\vec b}_i^H U_i & \widehat{T}_{i+1}
	\end{bmatrix}\right)  \begin{bmatrix}
		U_i^H \vec v_i \\ 0
	\end{bmatrix}.
\end{equation*}
From this, we easily obtain
\begin{equation}
	\label{eqn:krylov-update--residual-update}
	\vec r_{i-1} = V_i \widetilde{\vec y}_i + V_{i} \widetilde{\vec w}_i + \begin{bmatrix}
		V_i & \widehat{Q}_{i+1}
	\end{bmatrix} \widetilde{\vec r}_i = V_i r(S_i) \vec v_i + \vec w_i + \vec r_i,
\end{equation}
where
\begin{equation*}
	\vec w_i = - V_i U_i r(U_i^H S_i U_i) U_i^H \vec v_i 
\end{equation*}
and
\begin{equation}
	\label{eqn:krylov-update--remainder-generic-iteration-end}
	\vec r_i = \begin{bmatrix}
		V_{i} & \widehat{Q}_{i+1}
	\end{bmatrix} \widetilde{\vec r}_i
	= \begin{bmatrix}
		V_i U_i  & \widehat{Q}_{i+1}
	\end{bmatrix}
	r\left(\begin{bmatrix}
		U_i^H S_i U_i & U_i^H \widetilde{\vec b}_i \vec e_1^H \\
		\vec e_1 \widetilde{\vec b}_i^H U_i & \widehat{T}_{i+1}
	\end{bmatrix}\right)  \begin{bmatrix}
		U_i^H \vec v_i \\ 0
	\end{bmatrix}.
\end{equation}
Note that if we replace $i$ with $i+1$ in \cref{eqn:krylov-update--remainder-generic-iteration-start} we obtain \cref{eqn:krylov-update--remainder-generic-iteration-end}, i.e., we have written the remainder $\vec r_i$ in a form that is ready for the $(i+1)$-th cycle.

The approximate solution to $r(A) \vec b$ is updated with the identity
\begin{equation}
	\label{eqn:generic-cycle-approximation}
	\vec y_i = \vec y_{i-1} + \vec w_{i-1} + V_i r(S_i) \vec v_i.
\end{equation}  
Recalling \cref{eqn:krylov-update--first-iteration-solution-identity}, in the first cycle we have $\widehat{\vec y} = \vec y_1 + \vec w_1 + \vec r_1$. It is easy to prove by induction that a similar identity holds in all subsequent cycles. Indeed, assuming that $\widehat{\vec y} = \vec y_{i-1} + \vec w_{i-1} + \vec r_{i-1}$, we have
\begin{equation*}
	\widehat{\vec y} = \vec y_{i-1} + \vec w_{i-1} + V_i r(S_i) \vec v_i + \vec w_i + \vec r_i = \vec y_i + \vec w_i + \vec r_i,
\end{equation*} 
where we have used \cref{eqn:krylov-update--residual-update} and \cref{eqn:generic-cycle-approximation}.
Note that, similarly to the first cycle, in \cref{eqn:krylov-update--remainder-generic-iteration-end} the matrix~$V_i$ only appears multiplied by $U_i$, so after computing~$\vec y_i$ and~$\vec w_i$ it is no longer necessary to store the whole basis $V_i$.

\subsection{Final cycle}
\label{subsec:algorithm--final-cycle}

The final cycle is slightly simpler, since we can compute the remainder directly instead of using the low-rank update formula.
Indeed, at the beginning of the $s$-th cycle the remainder~$\vec r_{s-1}$ can be computed directly from \cref{eqn:krylov-update--remainder-generic-iteration-end}, that reads
\begin{equation*}
	\vec r_{s-1} = \begin{bmatrix}
		V_{s-1} U_{s-1}  & Q_s
	\end{bmatrix}
	r\left(\begin{bmatrix}
		U_{s-1}^H S_{s-1} U_{s-1} & U_{s-1}^H \widetilde{\vec b}_{s-1} \vec e_1^H \\
		\vec e_1 \widetilde{\vec b}_{s-1}^H U_{s-1} & T_s
	\end{bmatrix}\right)  \begin{bmatrix}
		U_{s-1}^H \vec v_{s-1} \\ 0
	\end{bmatrix}.
\end{equation*}
Using the same notation introduced in previous cycles, we define $V_s := \begin{bmatrix}
	V_{s-1} U_{s-1} & Q_s
\end{bmatrix}$,
\begin{equation*}
	\vec v_{s} := \begin{bmatrix}
		U_{s-1}^H \vec v_{s-1} \\
		0
	\end{bmatrix} \qquad \text{and} \qquad
	S_s := \begin{bmatrix}
		U_{s-1}^H S_{s-1} U_{s-1} & U_{s-1}^H \widetilde{\vec b}_{s-1} \vec e_1^H \\
		\vec e_1 \widetilde{\vec b}_{s-1}^H U_{s-1} & T_s
	\end{bmatrix},
\end{equation*} 
so we have
\begin{equation*}
	\vec r_{s-1} = V_s r(S_s) \vec v_s,
\end{equation*} 
and the final approximation to $r(A) \vec b$ is obtained as
\begin{equation}
	\label{eqn:final-cycle-approximation}
	\vec y_s = \vec y_{s-1} + \vec w_{s-1} + V_s r(S_s) \vec v_s.
\end{equation}
Note that \cref{eqn:final-cycle-approximation} coincides with \cref{eqn:generic-cycle-approximation} where $i$ has been replaced by $s$, so in the final cycle of the algorithm we compute the same update as in the other cycles, even though the derivation is different.

Since in the $(s-1)$-th cycle we have $\widehat{\vec y} = \vec y_{s-1} + \vec w_{s-1} + \vec r_{s-1}$, it follows that $\vec y_s = \widehat{\vec y}$, i.e., in the last cycle the algorithm that we just described computes the same approximation to~$r(A) \vec b$ as~$M = k + sm$ iterations of the outer Krylov subspace method. We are going to show in \cref{prop:algorithm-coincides-with-lanczos} that the same approximations as Lanczos are computed also in the intermediate cycles.
The resulting algorithm is summarized in \cref{algorithm:lowmemory-rational-lanczos-matfun}.

\begin{algorithm}
	\begin{algorithmic}[1]
		\Require{$A \in \C^{n \times n}, \vec b \in \C^{n}$, $k, s, m \in \N$, $\vec \xi_k = [\xi_0, \dots, \xi_{k-1}], \vec \theta_M = [\theta_0, \dots, \theta_{M -1}]$ such that $M = k + sm$, $\theta_0 = \infty$ and $\theta_{k + im} = \infty$ for $i = 1, \dots, s-1$.}
		\Ensure{ $\vec y_s \approx f(A) \vec b$, such that $\vec y_s \in \rat(A, \vec b, \vec \theta_M)$.}
		\LineComment{First outer cycle:}
		\State Run $m+k+1$ iterations of the rational Lanczos algorithm with the outer poles $\vec \theta_M$ to compute $Q_1$, $T_1$ and~$\vec b_1$
		\State Compute the first approximation $\vec y_1 = Q_1 f(T_1) \vec e_1 \norm{\vec b}_2$ 
		\State Define $S_1 := T_1$, $V_1 := Q_1$, $\vec v_1 := \vec e_1 \norm{\vec b}_2$ and $\widetilde{\vec b}_1 := \vec b_1$    
		\LineComment{Other outer cycles:}
		\For{$i = 2, \dots, s$}
		\State Construct the basis $U_{i-1}$ of the rational Krylov subspace $\rat(S_{i-1}, \widetilde{\vec b}_{i-1}, \vec \xi_k)$ 
		\State Compute the correction $\vec w_{i-1} = - V_{i-1} U_{i-1} f(U_{i-1}^H S_{i-1} U_{i-1}) U_{i-1}^H \vec v_{i-1}$ 
		\parState{Run $m$ additional iterations of the rational Lanczos algorithm with the outer poles $\vec \theta_M$ to compute $Q_i$, $T_i$ and~$\widetilde{\vec b}_i$}
		\State Define $V_i := \begin{bmatrix}
			V_{i-1} U_{i-1} & Q_i
		\end{bmatrix}$, $\vec v_i = \begin{bmatrix}
			U_{i-1}^H \vec v_{i-1} \\
			0
		\end{bmatrix}$, $S_i = \begin{bmatrix}
			U_{i-1}^H S_{i-1} U_{i-1} & U_{i-1}^H \widetilde{\vec b}_{i-1} \vec e_1^H \\
			\vec e_1 \widetilde{\vec b}_{i-1}^H U_{i-1} & T_i \\		
		\end{bmatrix}$.
		\State Update the approximation $\vec y_i = \vec y_{i-1} + \vec w_{i-1} + V_i f(S_i) \vec v_i$ 
		\EndFor
	\end{algorithmic}
	
		\caption{RK-Compressed rational Lanczos for $f(A) \vec b$} \label{algorithm:lowmemory-rational-lanczos-matfun}
	
\end{algorithm}
	
\section{Analysis and comparison with existing algorithms}
\label{sec:algorithm-analysis-and-comparison}

In this section we analyze \cref{algorithm:lowmemory-rational-lanczos-matfun} from both a theoretical and computational point of view, and we compare it with other low-memory methods from the literature.

\subsection{Theoretical results}

We start by showing that the iterates computed by \cref{algorithm:lowmemory-rational-lanczos-matfun} coincide with iterates of the rational Lanczos method associated to the Krylov subspace $\rat(A, \vec b, \vec \theta_M)$ when $f$ is a rational function. 

\begin{proposition}
	\label{prop:algorithm-coincides-with-lanczos}
	When $f$ is a rational function of type $(k-1, k)$ with poles given by~$\vec \xi_k$, the approximations $\{\vec y_i\}_{i = 1}^s$ computed by \cref{algorithm:lowmemory-rational-lanczos-matfun} coincide with the approximations given by~\cref{eqn:krylov-matfun-approx} with an outer Krylov subspace of the appropriate dimension. Precisely, for any $i = 1, \dots, s$ we have
	\begin{equation*}
		\vec y_i = \mat Q_{k + im} f(\mat T_{k + im}) \vec e_1 \norm{\vec b}_2, \qquad \mat T_{k + im} = \mat Q_{k + im}^H A \mat Q_{k + im},
	\end{equation*}
	where $\mat Q_{k + im}$ is the orthonormal basis of $\rat(A, \vec b, \vec \theta_{k + im})$ generated by the rational Lanczos algorithm.
\end{proposition}
\begin{proof}
	We have already observed that $\vec y_1$ coincides with the approximation~\cref{eqn:krylov-matfun-approx} from the Krylov subspace $\rat(A, \vec b, \vec \theta_{k+m})$, and we have shown at the end of~\cref{subsec:algorithm--final-cycle} that $\vec y_s$ coincides with the approximation $\widehat{\vec y}$ computed by the outer Krylov method after $M = k + sm$ iterations. 
	To show that this also holds for all other cycles, let us fix $1 < i < s$ and suppose that we run a variant of \cref{algorithm:lowmemory-rational-lanczos-matfun} with $s$ replaced by $s' = i$ and $M$ replaced by $M' = k + im$. 
	It is easy to see that this variant of the algorithm performs exactly the same operations as the original one up to the $(i-1)$-th cycle, and hence computes the same approximate solutions~$\vec y_1, \dots, \vec y_{i-1}$. The only difference from the original algorithm is that the $i$-th iteration is carried out by directly computing the residual $\vec r_{i-1}$ as described in \cref{subsec:algorithm--final-cycle}, instead of performing the low-rank update with \cref{cor:lowrank-update-matfun--eye}. 
	However, as shown in \cref{subsec:algorithm--final-cycle} for the final iteration of the original algorithm, the update formula \cref{eqn:final-cycle-approximation} for the final approximation to $r(A) \vec b$ coincides with~\cref{eqn:generic-cycle-approximation}, so this algorithm variant also computes the same approximation $\vec y_i$ as the original one. 
	Since the final approximation computed by the variant of the algorithm coincides with the approximation generated by the outer Krylov subspace $\rat(A, \vec b, \vec \theta_{k + im})$ after $M' = k + im$ iterations, we conclude that this also holds for the approximation~$\vec y_i$ computed in the $i$-th cycle of the original algorithm. 
\end{proof}

When \cref{algorithm:lowmemory-rational-lanczos-matfun} is applied to a general function $f$, we have to use \cref{cor:lowrank-update-matfun--eye--general-f} instead of \cref{cor:lowrank-update-matfun--eye}, so in each cycle we introduce an error that depends on the approximation error of $f$ with rational functions. The following proposition provides a bound on the error of \cref{algorithm:lowmemory-rational-lanczos-matfun} with respect to the rational Lanczos algorithm in the general case.

\begin{proposition}
	\label{prop:lowmem-rational-lanczos-error-bound}
	Assume that \cref{algorithm:lowmemory-rational-lanczos-matfun} with $s$ cycles is applied to a function $f$ that is analytic on $\mathbb{W}(A)$, using inner poles $\vec \xi_k$. 
	The error of \cref{algorithm:lowmemory-rational-lanczos-matfun} with respect to the rational Lanczos algorithm is bounded by
	\begin{equation}
		\label{eqn:lowmem-rational-lanczos-error-bound}
		\norm{\widehat{\vec y} - \vec y_{s}}_2 \le 4(s-1)\norm{\vec b}_2 \min_{r \in \poly_{k-1}/q} \norm{f - r}_{\mathbb{W}(A)},
	\end{equation}
	where $q$ is a polynomial with roots given by the finite elements of $\vec \xi_k$.  
\end{proposition}
\begin{proof}
	For convenience, we introduce the notation $\widetilde{S}_i = \begin{bmatrix}
		S_i & \widetilde{\vec b}_i \vec e_1^H \\
		\vec e_1 \widetilde{ \vec b}_i^H & \widehat{T}_{i+1}
	\end{bmatrix}$. 
	When we replace \cref{cor:lowrank-update-matfun--eye} with \cref{cor:lowrank-update-matfun--eye--general-f} in the $i$-th cycle of \cref{algorithm:lowmemory-rational-lanczos-matfun}, we introduce an additional error term $E_i(f)$ in the expression for $f(\widetilde{S}_i)$, that by \cref{eqn:lowrank-update-matfun--general-f-error} satisfies
	\begin{equation*}
		\norm{E_i(f)}_2 \le 4 \min_{r \in \poly_{k-1}/q} \norm{f - r}_{\mathbb{W}(\widetilde{S}_i)}.
	\end{equation*}
	Note that for all $i$, we have~$\mathbb{W}(\widetilde{S}_i) \subset \mathbb{W}(A)$.  
	The expression for $\vec r_{i-1}$ with the error term included becomes
	\begin{equation*}
		\vec r_{i-1} = V_i f(S_i) \vec v_i + \vec w_i + \vec r_i + \vec \epsilon_i,
	\end{equation*} 
	where
	\begin{equation*}
		\vec \epsilon_i = \begin{bmatrix}
			V_i & \widehat{Q}_{i+1}
		\end{bmatrix} E_i(f) \begin{bmatrix}
			\vec v_i \\
			0
		\end{bmatrix}.
	\end{equation*}
	Since $V_i$ and $\widetilde{Q}_{i+1}$ have orthonormal columns and $\norm{\vec v_i}_2 \le \norm{\vec b}_2$, we have
	\begin{equation*}
		\norm{\vec \epsilon_i}_2 \le \norm{E_i(f)}_2 \norm{\vec b}_2 \le 4 \norm{\vec b}_2 \min_{r \in \poly_{k-1}/q} \norm{f - r}_{\mathbb{W}(A)}.
	\end{equation*}    
	It is easy to see that, using the update formula \cref{eqn:generic-cycle-approximation} for $\vec y_i$ as in the rational function case, the approximate solution computed in the $i$-th cycle satisfies the identity
	\begin{equation*}
		\widehat{\vec y} = \vec y_{i} + \vec w_{i} + \vec r_{i} + \sum_{\ell = 1}^{i} \vec \epsilon_\ell.
	\end{equation*}  
	Since in the final cycle the remainder $\vec r_{s - 1}$ is computed directly, the total error of \cref{algorithm:lowmemory-rational-lanczos-matfun} with respect to the rational Lanczos algorithm is thus
	\begin{displaymath}
		\norm{\widehat{\vec y} - \vec y_{s}}_2 \le \sum_{\ell = 1}^{s-1} \vec \epsilon_i \le 4(s-1) \norm{\vec b}_2 \min_{r \in \poly_{k-1}/q} \norm{f - r}_{\mathbb{W}(A)}.
	\end{displaymath}
\end{proof}
\cref{prop:lowmem-rational-lanczos-error-bound} shows that if we take $\vec \xi_k$ as the poles of a highly accurate rational approximant of $f$, the iterates $\vec y_i$ generated by \cref{algorithm:lowmemory-rational-lanczos-matfun} essentially coincide with the corresponding approximation computed by the rational Lanczos algorithm, since the error in \cref{eqn:lowmem-rational-lanczos-error-bound} is going to be negligible compared to the error $\norm{f(A) \vec b - \widehat{\vec y}}_2$ of the outer Krylov subspace method.

\subsection{Computational discussion}

In its $i$-th cycle, \cref{algorithm:lowmemory-rational-lanczos-matfun} has to keep in memory the last two columns of $Q_{i-1}$ required to run the rational Lanczos algorithm for the computation of~$Q_i$, the compressed basis $V_i \in \C^{n \times (k+m)}$ and the projected matrix $S_i \in \C^{(k+m) \times (k+m)}$, as well as the vectors $\vec y_{i-1}$, $\vec w_{i-1}$, $\vec v_i$ and $\widetilde{\vec b}_i$.

Since $\theta_{k+(i-1)m} = \infty$, the projected matrix $T_i$ can be computed by exploiting the block structure of the matrices $\underline{\mat H}_{k + im}$ and $\underline{\mat K}_{k + im}$, as shown in \cref{subsec:krylov-background}. 
In particular, all the entries of $T_i$ except for the one in position $(1, 1)$ can be computed during the iterations of the rational Lanczos algorithm in which~$Q_i$ is constructed. 
The top-left entry of $T_i$ needs to be corrected according to~\cref{eqn:projected-matrix-identity--infinite-pole}, which requires us to compute~$\vec e_{k+(i-1)m} \mat K_{k+(i-1)m}^{-1} \vec e_{k+(i-1)m}$. 
Since we also have $\theta_{k+(i-2)m} = \infty$, the resulting block structure implies that we can compute $\mat K_{k+(i-1)m}^{-1} \vec e_{k+(i-1)m}$ by considering only the trailing principal $m \times m$ block of $\mat K_{k+(i-1)m}$, which is computed in the $(i-1)$-th cycle. 
As a consequence, the correction for the $(1, 1)$ entry of $T_i$ in~\cref{eqn:projected-matrix-identity--infinite-pole} can be precomputed during the $(i-1)$-th cycle, without the need to store the whole matrices $\underline{\mat H}_{k + im}$ and $\underline{\mat K}_{k + im}$. We omit the technical details.

It is important to note that \cref{algorithm:lowmemory-rational-lanczos-matfun} only computes functions of matrices of size~$k \times k$ and $(k+m) \times (k+m)$, that are independent on the number of cycles $s$, so the cost of computing functions of projected matrices does not increase with the iteration number, in contrast with the rational Lanczos method. This can lead to significant computational savings when the number of iterations is very large.

In each cycle, \cref{algorithm:lowmemory-rational-lanczos-matfun} also has to construct a $k$-dimensional rational Krylov subspace associated with a $(k+m) \times (k+m)$ matrix. Since $k$ and $m$ are typically much smaller than~$n$, the cost of operations with  matrices and vectors of size $m+k$ is usually negligible with respect to the cost of operations with vectors of size $n$, so we expect that the construction of the inner rational Krylov subspace will not have a significant impact on the overall performance of the method.

In addition to the parameter $k$, which determines the accuracy of the inner rational approximation, and hence the highest accuracy attainable by the method, we also have to choose the parameter $m$, which determines the number of vectors that are added to the basis between two consecutive compressions. This parameter should be chosen according to the available memory. Taking $m$ small decreases memory requirements but increases the frequency of computations with $(k+m)\times(k+m)$ matrices. 

Note that in the pseudocode of \cref{algorithm:lowmemory-rational-lanczos-matfun}, our method only computes the approximate solutions $\vec y_i$ once every $m$ iterations. 
However, the algorithm can be easily adapted to compute an approximate solution in every outer iteration that employs an infinite pole, using the same update formula as in line 9 of \cref{algorithm:lowmemory-rational-lanczos-matfun} but with a smaller matrix $S_i$. 
Note that the only difference with respect to the definition of $S_i$ in line 8 is in the $T_i$ block and in the size of the off-diagonal blocks, so the only additional operation required to compute an approximate solution is the computation of a matrix function of size at most $(k+m) \times (k+m)$. 
It is easy to see that an approximate solution computed in this way also coincides with the one constructed in the corresponding iteration of the outer Krylov method, via the same argument used in the proof of \cref{prop:algorithm-coincides-with-lanczos}. 
This technique is employed in our implementation of \cref{algorithm:lowmemory-rational-lanczos-matfun}.

Instead of running the method for a fixed number of cycles, in practice it is desirable to run it until we obtain a solution with a certain accuracy. A commonly used stopping criterion that can be employed for this purpose is the norm of the difference of two consecutive approximations. This simple criterion offers no guarantee on the final error and it may underestimate it in practice, especially when convergence is slow, but we found it to be accurate enough for our purposes. 
Observe that in order to check the convergence condition it is not necessary to form the approximate solutions~$\vec y_i$ of length $n$, but it is enough to compute differences of short vectors. Indeed, we have
\begin{align*}
	\vec y_i - \vec y_{i-1} &= \vec w_{i-1} + V_i f(S_i) \vec v_i \\
	&= \begin{bmatrix}
		V_{i-1} U_{i-1} & Q_i
	\end{bmatrix} \left(\begin{bmatrix}
		f(U_{i-1}^H S_{i-1} U_{i-1}) U_{i-1}^H \vec v_{i-1} \\
		0
	\end{bmatrix}+ f(S_i) \vec v_i \right),
\end{align*} 
and since $V_i$ has orthonormal columns we conclude that
\begin{equation*}
	\norm{\vec y_i - \vec y_{i-1}}_2 = \norm*{\begin{bmatrix}
		f(U_{i-1}^H S_{i-1} U_{i-1}) U_{i-1}^H \vec v_{i-1} \\
		0
	\end{bmatrix}+ f(S_i) \vec v_i }_2.
\end{equation*} 
This formulation allows us to check if the stopping criterion is satisfied without performing operations with vectors of length $n$, which results in some computational savings. Our implementation of \cref{algorithm:lowmemory-rational-lanczos-matfun} employs this expression in its stopping criterion.

\subsection{Comparison with other low-memory Krylov methods}

In this section we briefly compare our method with other low-memory Krylov subspace methods from the literature, highlighting advantages and disadvantages of each method.
Similarly to multishift CG \cite{frommer99fast}, \cref{algorithm:lowmemory-rational-lanczos-matfun} is based on a rational approximant, so the attainable accuracy is ultimately limited by the available memory, since the number of vectors that must be stored is	proportional to the number of poles used in the rational approximation. 
However, while multishift CG requires an explicit rational approximant in partial fraction form, our method only needs the poles of the rational function. This makes \cref{algorithm:lowmemory-rational-lanczos-matfun} easier to use, and less susceptible to numerical errors caused by the representation of the rational function.

Given the similarities between \cref{algorithm:lowmemory-rational-lanczos-matfun} and multishift CG, it is useful to compare the cost of the two methods in more detail. 
For this purpose, we employ the simple computational model used in \cite[Experiment~5.3]{guettel21comparison}, in which operations on vectors of length $n$ such as scaling or addition are counted as one unit of work, denoted by $1 \mathcal{V}$, and operations with vectors or matrices of size independent of $n$ are not counted. 
As discussed in \cite{guettel21comparison}, when the underlying rational approximant has $k$ poles the multishift CG algorithm must store $2 k$ vectors of length $n$ and perform approximately $5 k \mathcal{V}$ operations in addition to the standard Lanczos algorithm. 
\cref{algorithm:lowmemory-rational-lanczos-matfun} with $k$ inner poles and compression every $m$ outer iterations must store $k+m$ vectors of length $n$, and the only operation that involves vectors of length $n$ outside of the standard Lanczos algorithm is in the compression step (line 8 in \cref{algorithm:lowmemory-rational-lanczos-matfun}), where the product $V_{i-1} U_{i-1}$ has to be computed for a cost of~$2k(k+m) \mathcal{V}$ every $m$ iterations, that on average amounts to~$2k m^{-1} (k+m) \mathcal{V}$ at each iteration. 
If we take~$m = k$, so that \cref{algorithm:lowmemory-rational-lanczos-matfun} and multishift CG have the same memory requirements and the same attainable accuracy, the cost of \cref{algorithm:iter-short-rec-arnoldi} is $4 k \mathcal{V}$ per iteration, which is smaller than the $5 k \mathcal{V}$ cost of multishift CG. 
However, note that the efficiency of multishift CG can be improved by removing converged linear systems, i.e., by no longer updating the approximate solutions of linear systems when the residual becomes smaller than the requested tolerance \cite[Section~5.3]{vandeneshof02numerical}. 
This can significantly reduce the cost of the method, since linear systems associated to different poles often have substantially different convergence rates.

In contrast with \cref{algorithm:lowmemory-rational-lanczos-matfun} and multishift CG, the memory requirements of the two-pass Lanczos method \cite{borici00fast} are independent of the target accuracy, with the exception of the projected matrix, which grows in size at each iteration. 
When many iterations are needed to reach convergence, the computation of functions of projected matrices of increasing size can have a significant impact on the performance of the method, although this can be mitigated by only computing the approximate solution once every~$d > 1$ iterations. For the same number of Lanczos iterations, the two-pass version requires twice the number of matrix-vector products with $A$ compared to the other methods. 
However, in practice the cost is usually less than doubled, because in the second pass it is only necessary to recompute the Krylov basis vectors, and the orthogonalization coefficients have already been computed in the first pass.

The restarted Krylov method for Stieltjes matrix functions \cite{frommer14efficient,frommer14convergence} requires storing a number of vectors proportional to the restart length. Although the amount of memory available does not influence the accuracy attainable by this method, a shorter restart length can cause delays in the convergence, similarly to what happens when restarting Krylov subspace methods for the solution of linear systems. The convergence delay can be mitigated by employing deflation techniques~\cite{eiermann11deflated}.
We note that this restarted method explicitly requires the expression of the integrand function in the Stieltjes representation of $f$. The restarted method can also be applied to a function that is not Stieltjes by using a different integral representation, such as one based on the Cauchy integral formula. This was done in \cite[Section~4.3]{frommer14efficient} for the exponential function.

\section{Numerical experiments}
\label{sec:numerical-experiments}

In this section we compare \cref{algorithm:lowmemory-rational-lanczos-matfun}, which we denote by {\tt \nostro}, with other Krylov subspace methods for the computation of the action of a matrix function on a vector. 
For all methods, we monitor the relative norm of the difference between two consecutive computed solutions and stop when it becomes smaller than a requested tolerance. This difference is monitored at every iteration in which an infinite pole is used.

The MATLAB code to reproduce the experiments in this section is available on Github at \url{https://github.com/casulli/ratkrylov-compress-matfun}. We use our own implementation of \cref{algorithm:lowmemory-rational-lanczos-matfun}, two-pass Lanczos and multishift CG, while for the restarted Arnoldi for Stieltjes matrix functions we use the {\tt funm\_quad} implementation~\cite{schweitzer14funmquad,frommer14efficient}. Although our implementation of \cref{algorithm:lowmemory-rational-lanczos-matfun} has no external dependencies, to compute the rational approximant in partial fraction form employed in the multishift CG algorithm we use the implementation of the AAA algorithm \cite{nakatsukasa18aaa} from the chebfun package \cite{driscoll14chebfun}. All the experiments were performed with MATLAB R2021b on a laptop running Ubuntu 20.04, with 32 GB of RAM and an Intel Core i5-10300H CPU with clock rate 2.5 GHz, using a single thread.

\subsection{Exponential function}
\label{subsec:experiments--exponential}

As a first test problem, we consider the computation of~$e^{-tA}\vec  1$, where the matrix $A\in \C^{n^2\times n^2}$ is the discretization of the 2D Laplace operator with zero Dirichlet boundary conditions using centered finite differences with $n+2$ points in each direction, that is
\begin{equation}
	\label{eqn:2d-laplacian-definition}
	A= B\otimes I + I \otimes B, \quad \text{where} \quad B=(n+1)^{2}\left[\begin{smallmatrix}
		2 & -1\\
		-1 & \ddots & \ddots\\
		& \ddots & \ddots & -1\\
		&& -1 & 2
	\end{smallmatrix}\right]\in \C^{n\times n},
\end{equation}
and $\vec 1\in \C^{n^2}$ is the vector of all ones. We compare the accuracy and timing of different methods based on Krylov subspaces using as a reference solution  ${ e^{-tB}}\vec 1 \otimes { e^{-tB}}\vec 1$, where the involved exponentials are computed using the MATLAB command {\tt expm}. We set $n = 10^3$ (therefore the size of $A$ is $10^6 \times 10^6$) and choose different values of $t$. As $t$ increases, the eigenvalues of the matrix $-t A$ of which we compute the exponential become more spread out, so we expect an increasing number of iterations for the convergence of polynomial Krylov subspace methods (see, e.g., \cite[Section~4]{beckermann2009error}).

We use \cref{algorithm:lowmemory-rational-lanczos-matfun} with all outer poles equal to $\infty$ and $k = 25$ inner poles of the form described in \cite{carpenter1984extended}, which guarantee a rational approximation of $e^x$ for $x \in (-\infty,0]$ with an absolute error of the order of machine precision. In particular, assuming that the involved matrix has no positive eigenvalues, the inner poles are independent of the spectrum of the matrix.
We compare our {\tt \nostro} method with the standard Lanczos algorithm ({\tt lanczos}), the two-pass version of Lanczos ({\tt lanczos-2p}) and the Arnoldi algorithm with full orthogonalization ({\tt arnoldi}). Note that both {\tt lanczos} and {\tt arnoldi} require storing the whole Krylov basis.

In Figure~\ref{fig:exponential-time}, we report the time needed to compute $e^{-tA}\vec 1$ with the different methods for different values of $t$, using a relative tolerance of~$10^{-10}$ in the stopping criterion. 
For any fixed value of $t$, all the employed methods converge in the same number of iterations, which is reported in \cref{table:exponential-iter} along with the final relative error attained. All the methods stop at the same iteration and produce the same approximate solution, confirming that \cref{algorithm:lowmemory-rational-lanczos-matfun} is reproducing the convergence of the Lanczos algorithm (up to the error in the approximation of $e^x$ with a rational function, which is negligible in this case). As the required number of iterations increases, our {\tt \nostro} method appears to be the fastest. This is mainly due to the fact that in \cref{algorithm:lowmemory-rational-lanczos-matfun} the size of the matrix functions computed during the execution of the algorithm does not increase with the size of the Krylov subspace, in contrast with the other methods. 

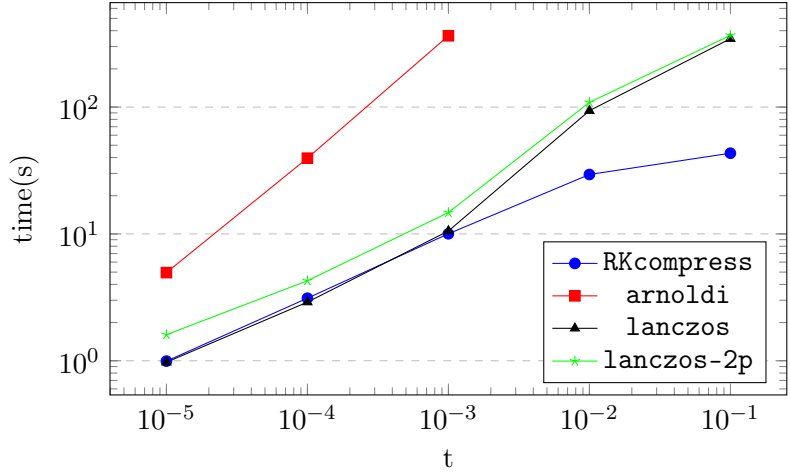
\begin{figure}
	\makebox[\linewidth][c]{
		\begin{tikzpicture}
			\begin{loglogaxis}[
				title = {},  
				xlabel = {t},
				ylabel = {time(s)},
				x tick label style={/pgf/number format/.cd,%
					scaled x ticks = false,
					set thousands separator={},
					fixed},
				legend pos=south east,
				ymajorgrids=true,
                height=.45\textwidth,
				grid style=dashed]
		
				\addplot[color=blue, mark=*, mark size=2pt] table{exp_low-mem.dat};
				\addlegendentry{{\tt{\nostro}}}

				\addplot[ color=red, mark=square*, mark size=2pt] table{exp_Arnoldi-full.dat};
				\addlegendentry{{\tt{arnoldi}}}

                \addplot[color=black, mark=triangle*, mark size=2pt] table {exp_Lanczos.dat};
				\addlegendentry{{\tt{lanczos}}}

                \addplot[,color=green, mark=star, mark size=2pt] table{exp_Lanczos-twoPass.dat};
				\addlegendentry{{\tt{lanczos-2p}}}
			\end{loglogaxis}
	\end{tikzpicture}}
\caption{Time needed for the computation of $e^{-tA}\vec 1$ with accuracy of $10^{-10}$, for different values of $t$ and employing different methods.}\label{fig:exponential-time}
\end{figure}

\begin{table}
    \centering
	\caption{Number of iterations and final relative error for all the methods in Figure~\ref{fig:exponential-time}, using relative tolerance $10^{-10}$.}
	\label{table:exponential-iter}
	\begin{tabular}{c|rrrrrr}
		\toprule
		$t$ & &$10^{-5}$ & $10^{-4}$ & $10^{-3}$ &$10^{-2}$ & $10^{-1}$\\
		\midrule
		iter & & 39 & 119 & 372 & 1104 & 1650 \\
		err & & 3.98e-11 & 1.89e-10 & 6.54e-10 & 2.26e-09 & 3.01e-09 \\
		\bottomrule
	\end{tabular}
\end{table}

\subsection{Markov functions}
\label{subsec:experiments--markov-functions}

An important class of functions for which rational approximations have been already described in the literature \cite{beckermann2009error} is given by Markov functions, which can be represented as 
\begin{equation}\label{eqn:Markov_fun}
    f(z)=\int_{\alpha}^{\beta} \frac{d\mu(x)}{z-x},
\end{equation}
where $\mu(x)$ is a positive measure and $-\infty \le \alpha < \beta < \infty.$ This class contains several frequently used functions, such as
\begin{equation*}
    \frac{\log(1 + z)}{z}=-\int_{-\infty}^{-1}\frac{1/x}{z-x}dx \quad \mathrm { and } \quad z^\gamma=\frac{\sin(\pi\gamma)}{\pi}\int_{-\infty}^{0}\frac{|x|^\gamma}{z-x}dx,
\end{equation*}
with $-1<\gamma<0$.
If $f$ is defined as in \eqref{eqn:Markov_fun} and $A$ is a Hermitian matrix with eigenvalues greater than~$\beta$, we can use the approach in \cite{beckermann2009error} to choose quasi-optimal inner poles for {\tt \nostro}. In particular, to obtain a rational approximation of $f$ on an interval $[a,b]$ with $a > \beta$ with relative error norm bounded by $\epsilon$ it is sufficient to use $k$ poles where
\begin{equation}
	\label{eqn:bound-markov}
	k\ge \log\left(\frac{4}{\epsilon}\right)\frac{\log(16(b-\beta)/(a-\beta))}{\pi^2},
\end{equation}
therefore the number of poles needed in {\tt \nostro} depends logarithmically on the condition number of $A-\beta I$. We refer to \cite[Section~6.2]{beckermann2009error} for more details regarding the choice of poles.

In the experiment that follows, we compute $A^{-1/2} \vec 1$, where $A \in \C^{n^2 \times n^2}$ is a discretization of the 2D Laplace operator \cref{eqn:2d-laplacian-definition} with increasing $n = 200, 400, \dots, 1000$, comparing several low-memory Krylov subspace methods, using a relative tolerance of $10^{-8}$. The reference solution is computed by diagonalizing $A$, exploiting the Kronecker sum structure.  
We compare \cref{algorithm:lowmemory-rational-lanczos-matfun} ({\tt \nostro}) with the two-pass Lanczos method ({\tt lanczos-2p}), the standard multishift CG method ({\tt msCG}), a more efficient implementation of multishift CG with removal of converged linear systems ({\tt msCG-rem}), and the restarted Krylov method for Stieltjes functions, both with and without deflation, denoted by {\tt restart-defl} and {\tt restart}, respectively.

For \cref{algorithm:lowmemory-rational-lanczos-matfun}, we use outer poles all equal to $\infty$ and inner poles from \cite{beckermann2009error} with $k$ given by~\cref{eqn:bound-markov} and $m = k$. For a tolerance of $10^{-8}$, the values of $k$ range from $k = 26$ to $k = 32$ for the different matrix sizes. The rational approximant for multishift CG is obtained by running the AAA algorithm~\cite{nakatsukasa18aaa} with tolerance~$10^{-12}$ on a discretization of the spectral interval of $A$; this produces a rational approximant with either $18$ or $19$ poles for all matrix sizes, that are fewer than the ones obtained using \cref{eqn:bound-markov}, but it does not provide any theoretical guarantee on the approximation error. For the restarted Krylov method for Stieltjes functions, we set a restart length equal to $2k$ (so that the memory requirement is the same as {\tt \nostro}) and we use the default options in {\tt funm\_quad}; we retain $5$ Ritz vectors in the variant with deflation.

\begin{figure}[htbp]
	\makebox[\linewidth][c]{
		\begin{tikzpicture}
			\begin{semilogyaxis}[
				title = {},  
				xlabel = {$n$},
				ylabel = {time (s)},
				xtick={200, 400, 600, 800, 1000}, 
				x tick label style={/pgf/number format/.cd,%
					scaled x ticks = false,
					set thousands separator={},
					fixed},
				legend pos=south east,
				ymajorgrids=true,
                height=.45\textwidth,
				grid style=dashed]
		
				\addplot[color=blue, mark=*, mark size=2.5pt] table{invsqrt_low-memory.dat};
				\addlegendentry{{\tt{\nostro}}}

				\addplot[ color=red, mark=square*, mark size=2.5pt] table{invsqrt_lanczos-twopass.dat};
				\addlegendentry{{\tt{lanczos-2p}}}

                \addplot[color=black, mark=triangle*, mark size=2.5pt] table {invsqrt_multishiftCG.dat};
				\addlegendentry{{\tt{msCG}}}

                \addplot[color=brown, mark=star, mark size=2.5pt] table{invsqrt_multishiftCG-rem.dat};
				\addlegendentry{{\tt{msCG-rem}}}

				\addplot[color=violet, mark=diamond*, mark size=2.5pt] table{invsqrt_restarted.dat};
				\addlegendentry{{\tt{restart}}}

				\addplot[color=teal, mark=x, mark size=2.5pt] table{invsqrt_restarted-defl.dat};
				\addlegendentry{{\tt{restart-defl}}}

			\end{semilogyaxis}
	\end{tikzpicture}}
\caption{Time needed for the computation of $A^{-1/2}\vec 1$ with relative tolerance $10^{-8}$, where $A \in \C^{n^2 \times n^2}$ is a discretization of the 2D Laplace operator for increasing $n$, employing different low-memory methods.}
\label{fig:invsqrt-time-1}
\end{figure}
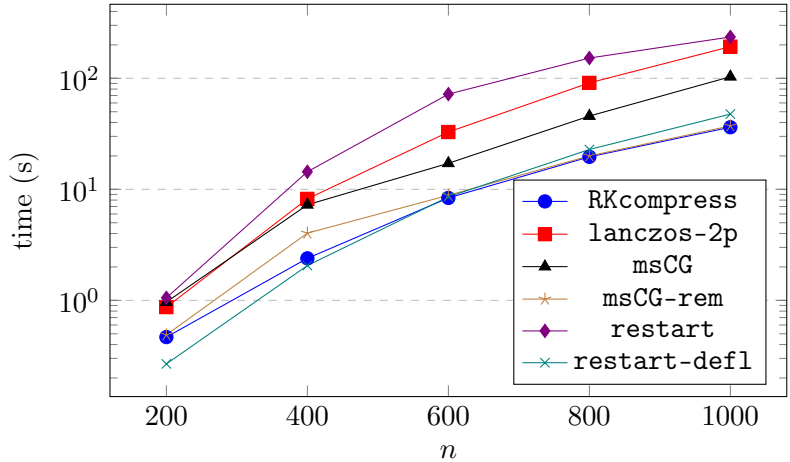

\begin{table}[htbp]
    \centering
	\caption{Final relative errors for the experiment in Figure~\ref{fig:invsqrt-time-1}.}
	\label{table:invsqrt-err-1}
	\begin{tabular}{r|cccccc}
		\toprule
		$\text{size}(A)$   &{\tt \nostro} & {\tt lanczos-2p} & {\tt msCG} & {\tt msCG-rem} & {\tt restart} & {\tt restart-defl} \\
		\midrule
		40000 & 9.01e-08 & 9.01e-08 & 7.81e-09 & 7.82e-09 & 2.93e-10 & 1.68e-11 \\
		160000 & 1.29e-07 & 1.29e-07 & 6.50e-09 & 6.50e-09 & 2.44e-08 & 1.58e-09 \\ 
		360000 & 1.70e-07 & 1.70e-07 & 2.68e-05 & 2.68e-05 & 5.29e-08 & 2.19e-09 \\
		640000 & 2.47e-07 & 2.47e-07 & 3.85e-06 & 3.85e-06 & 1.83e-05 & 4.34e-09 \\
		1000000 & 3.86e-07 & 3.86e-07 & 1.32e-06 & 1.32e-06 & 3.78e-04 & 2.52e-09 \\
		\bottomrule
	\end{tabular}
\end{table}

\begin{table}[htbp]
    \centering
	\caption{Number of iterations for the experiment in Figure~\ref{fig:invsqrt-time-1}.}
	\label{table:invsqrt-iter-1}
	\begin{tabular}{r|SSSSS[table-format=5.0]S}
		\toprule
		$\text{size}(A)$ &{\tt \nostro} & {\tt lanczos-2p} & {\tt msCG} & {\tt msCG-rem} & {\tt restart} & {\tt restart-defl} \\
		\midrule
		40000 & 282 & 282 & 379 & 379 & 2496 & 468 \\
		160000 & 554 & 554 & 747 & 747 & 8232 & 896 \\
		360000 & 823 & 823 & 1122 & 1122 & 16380 & 1440 \\
		640000 & 1085 & 1085 & 1497 & 1497 & 18600 & 2108 \\
		1000000 & 1336 & 1336 & 1872 & 1872 & 19200 & 2880 \\
		\bottomrule
	\end{tabular}
\end{table}

We report the times in \cref{fig:invsqrt-time-1}, the final errors in \cref{table:invsqrt-err-1} and the number of iterations in \cref{table:invsqrt-iter-1}. 
The methods with the shortest runtime are {\tt \nostro}, {\tt msCG-rem} and {\tt restart-defl}. Observe that {\tt restarted-defl} requires significantly more iterations compared to the methods without restarting, so it would perform worse in a situation in which matrix-vector products with $A$ are more expensive.  
Observe that for $n \ge 600$ the final error of multishift CG is significantly larger than the requested tolerance, suggesting that the rational approximant given by AAA is not accurate enough. However, it turns out that running {\tt \nostro} with inner poles given by the AAA approximant has roughly the same error as with the poles from \cite{beckermann2009error}, implying that the error in multishift CG should be mainly attributed to the explicit partial fraction representation of the rational function.

We note that the final error of {\tt restart-defl} is smaller compared to the errors of {\tt \nostro} and {\tt lanczos-2p}, since in the {\tt funm\_quad} implementation the approximate solutions are computed and compared only at the end of each restart cycle, hence the stopping criterion is more reliable and less likely to underestimate the error.
The restarted Krylov method without deflation was unable to converge to the requested accuracy within $300$ restart cycles for~$n \ge 800$.

\subsection{Rational outer Krylov method}
\label{subsec:experiments--shift-and-invert}
In this experiment we examine the effectiveness of \cref{algorithm:lowmemory-rational-lanczos-matfun} with a rational outer Krylov method. We let $A \in \C^{n \times n}$ be a discretization of the one-dimensional Laplace operator, and we approximate $A^{-1/2} \vec 1$ with stopping tolerance~$10^{-6}$, using as reference solution the approximation computed with a rational Krylov method with tolerance~$10^{-8}$. 
We use a shift-and-invert outer Krylov method with the optimal single pole given in \cite[Section~6.1]{beckermann2009error}. For the inner poles in \cref{algorithm:lowmemory-rational-lanczos-matfun}, we use the quasi-optimal poles from \cite{beckermann2009error} with~$k$ given by \cref{eqn:bound-markov} and we set $m = 10$. We compare {\tt \nostro} with the two-pass rational Lanczos algorithm ({\tt rat-lanczos-2p}) that uses the same poles (i.e., the shift-and-invert poles interleaved with an infinite pole at the iterations in which {\tt \nostro} performs a compression), so that the two methods have the same convergence rate, and they both check the convergence condition every $10$ iterations. The times and final errors for increasing values of $n$ are shown in \cref{table:shift-and-invert}. While both methods converge in the same number of iterations for all matrix sizes, {\tt \nostro} is faster than {\tt rat-lanczos-2p} since it needs to solve half the number of shifted linear systems with the matrix $A$ and it computes functions of smaller projected matrices.

\begin{table}[htbp]
    \centering
	\caption{Time, final error and number of iterations for the experiment in \cref{subsec:experiments--shift-and-invert}.}
	\label{table:shift-and-invert}
	\begin{tabular}{c|cc|cc|c}
		\toprule
		\multirow{2}{*}{$\text{size}(A)$} &\multicolumn{2}{c|}{time (s)} & \multicolumn{2}{c|}{error} & \multirow{2}{*}{iter} \\
		& {\tt \nostro} & {\tt rat-lanczos-2p} & {\tt \nostro} & {\tt rat-lanczos-2p}  \\
		\midrule
		\phantom{1}50000 & 0.43 & 0.47 & 3.38e-08 & 7.22e-09 & 177 \\
		100000 & 1.03 & 1.31 & 1.67e-07 & 7.28e-08 & 239 \\
		150000 & 1.89 & 2.53 & 1.52e-07 & 2.82e-07 & 290 \\
		200000 & 3.49 & 3.97 & 2.55e-07 & 1.05e-07 & 331 \\
		\bottomrule
	\end{tabular}
\end{table}

\subsection{Numerical loss of orthogonality}
\label{subsec:experiments--loss-of-orthogonality}

In order to investigate the behavior of \cref{algorithm:lowmemory-rational-lanczos-matfun} in finite precision arithmetic, we compute $e^A \vec b$, where $A \in \C^{2000 \times 2000}$ is a tridiagonal matrix with logspaced eigenvalues in the interval $[-10^{4}, -10^{-4}]$, and $\vec b$ is a random vector with unit normal entries.
The convergence of {\tt \nostro}, {\tt lanczos} and {\tt arnoldi} are compared in \cref{fig:lossorth}. Both {\tt lanczos} and {\tt \nostro} exhibit a delay in convergence due to the loss of orthogonality in the Krylov basis and they are essentially indistinguishable, so it appears that \cref{algorithm:lowmemory-rational-lanczos-matfun} also reproducess the finite precision behavior of the Lanczos algorithm.

\begin{figure}[htbp]
	\makebox[\linewidth][c]{
		\begin{tikzpicture}
			\begin{semilogyaxis}[
				title = {},  
				xlabel = {iteration},
				ylabel = {relative error},
				xtick={100, 200, 300, 400}, 
				x tick label style={/pgf/number format/.cd,%
					scaled x ticks = false,
					set thousands separator={},
					fixed},
				legend pos=south west,
				ymajorgrids=true,
                height=.35\textwidth,
				grid style=dashed]
		
				\addplot[very thick, dashed, color=blue, mark size=2.5pt] table{lossorth_low-mem.dat};
				\addlegendentry{{\tt{\nostro}}}

				\addplot[very thick, dashdotted, color=black, mark size=2.5pt] table{lossorth_lanczos.dat};
				\addlegendentry{{\tt{lanczos}}}

                \addplot[very thick, color=red, mark size=2.5pt] table {lossorth_arnoldi.dat};
				\addlegendentry{{\tt{arnoldi}}}

			\end{semilogyaxis}
	\end{tikzpicture}}
\caption{Effect of numerical loss of orthogonality in the computation of~$e^A \vec b$, where $A \in \C^{2000 \times 2000}$ has logspaced eigenvalues in $[-10^{4}, -10^{-4}]$ and~$\vec b$ is a random vector.}
\label{fig:lossorth}
\end{figure}
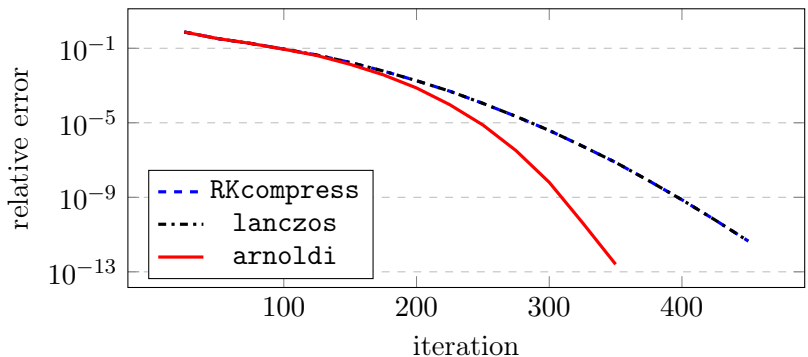

\section{Conclusions}
\label{sec:conclusions}

We have presented a memory-efficient method for the computation of $f(A) \vec b$ for a Hermitian matrix $A$. 
The method combines an outer rational Lanczos method with inner rational Krylov iterations, that are used to compress the Lanczos basis and reduce memory requirements. 
The construction of the inner rational Krylov basis only involves operations with small projected matrices and it does not require any operation with the matrix $A$. 
We have proved that our algorithm coincides with the outer Krylov method when $f$ is a rational function. In the general case, the error depends on a rational approximation of~$f$. 
Our numerical experiments show that the proposed algorithm is competitive with other low-memory methods based on polynomial Krylov subspaces. The algorithm can also be used effectively to reduce memory requirements of the shift-and-invert Lanczos method when it takes several iterations to converge.
The possibility of extending the proposed approach to non-Hermitian matrices remains an open problem.

\section*{Acknowledgements}
The authors are members of the research group INdAM-GNCS.

\bibliographystyle{siam}
\bibliography{manuscript-biblio}

\end{document}